\documentclass{amsart}
\usepackage{amsfonts}
\usepackage{amssymb}
\usepackage{color}

\newtheorem{thm}{Theorem}[section]
\newtheorem{cor}[thm]{Corollary}

\newtheorem{lem}[thm]{Lemma}
\newtheorem{prob}[thm]{Problem}
\newtheorem{prop}[thm]{Proposition}
\theoremstyle{definition}
\newtheorem{defn}{Definition}[section]
\newtheorem{exmp}{Example}[section]

\theoremstyle{remark}
\newtheorem{rem}{Remark}[section]

\numberwithin{equation}{section}

\begin{document}

\title[The Grothendieck duality and sparse minimizing
]
{The Grothendieck duality and sparse minimizing  in 
spaces of Sobolev solutions to elliptic systems}

\author{A.A. Shlapunov}
\email{ashlapunov@sfu-kras.ru}  

\author{A.N. Polkovnikov}
\email{paskaattt@yandex.ru} 

\author{K.V. Gagelgans}
\email{ksenija.sidorova2017@yandex.ru} 

\address[Alexander Shlapunov, Alexander Polkovnikov, Kseniya Gagelgans]{Siberian Federal University
                                                 \\
         pr. Svobodnyi 79
                                                 \\
         660041 Krasnoyarsk
                                                 \\
         Russia}



\begin{abstract}
We present an instructive example of using Banach spaces of solutions to (linear, generally, 
non-scalar) elliptic operator $A$ to investigate variational inverse problems related to neural networks 
and/or to regularization of solutions to boundary value problems. More precisely, inspired by 
kernel's method for optimization problems in locally convex spaces, we prove the  existence of the 
so-called sparse minimizers for the related variational problem  and produce a representer theorem where a 
suitable fundamental solution of the operator $A$ is used as a reproducing kernel. The Grothendieck 
type duality for the Sobolev spaces of solutions to elliptic operator $A$ plays an essential role in 
the considerations. The case where the number of data passes to infinity is also discussed. 
Some typical situations related to the standard elliptic operators, the 
corresponding function spaces and fundamental solutions are considered. 
\end{abstract}

\maketitle

\section*{Introduction}
\label{s.Int}

Theoretical justification of an effective work of neural networks  attracts attention 
of many mathematicians in recent years, see, for instance, 
\cite{PN2, SS, Un20}. 
A large part of the papers were devoted to kernel's approach and, in  particular, to the use of 
spaces with reproducing kernels for proving the existence of the so-called sparse minimizers for the 
related variational problem  and a representer theorem providing a general form of the minimizers in 
particular situations, \cite{BCCDGW, BreCar, Carm+, 
LZZ, PN1, SS, 
SZ13, Un23A, ZXZ}. The most 
constructive answers were obtained within the framework of a Hilbert space ${\mathcal H} (\Omega)$ of functions 
over a set $\Omega \subset {\mathbb R}^n$  with reproducing kernel $ {\mathcal K} (\cdot,\cdot)$ where Riesz theorem 
on the general form of linear functionals plays an essential role. More 
precisely, in this case a sparse  solution to the problem of minimizing a functional 
\begin{equation} \label{eq.min.H}
{\mathcal F}(u)=F({\mathcal M} u)  + \|u\|_{ {\mathcal H}(\Omega)} , \, u \in {\mathcal H} (\Omega)
\end{equation}
with a linear continuous mapping ${\mathcal M}: {\mathcal H} (\Omega) \to {\mathbb R}^N$ and a  convex coercive 
lower semi-continuous functional $F$ over  ${\mathbb R}^N$, is given by a linear combination related 
to the values of the reproducing kernel ${\mathcal K} (x^{(j)},\cdot)$ at distinguished points $x^{(1)}, 
\dots x^{(N)}$ of the training set $\Omega$. On the other hand, functional \eqref{eq.min.H} can be 
considered in a more general context as Tikhonov' regularization of the  ill-posed problem related to 
an operator $\tilde {\mathcal M}:  {\mathcal H} (\Omega) \to \tilde {\mathcal H} (\Omega)  $ and to the 
operator equation 
\begin{equation} \label{eq.ill-posed}
\tilde {\mathcal M} u = f_0
\end{equation} 
with given datum $f_0$ from a function space $\tilde {\mathcal H}(\Omega)$. This  is typical if one looks 
for approximate solutions to boundary value problems for partial differential equations, 
\cite{Tikh77}, \cite{Engl}, \cite{LvRShi}. 
Naturally, the summand $\|\cdot\|_{{\mathcal H}(\Omega)}$ is chosen to play a 
stabilising role for  problem \eqref{eq.ill-posed}.

Unfortunately, the results in Hilbert spaces are diminished by a shortage  of explicitly written 
reproducing kernels. Of course, one could write the reproducing kernel as a series constructed with the use of an 
orthonormal basis in the Hilbert space,  see \cite{Ar}, but the explicit formulas for the kernels are very rare. We 
should mention the Bergman space ${\mathcal O} (D)\cap L^2 (D)$ of holomorphic $L^2 (D)$ functions in a bounded domain $D$ in $n$-dimensional 
complex space ${\mathbb C}^n$ with the Bergman (reproducing) kernel, see \cite{Berg}, the Hardy space 
${\mathcal O}^2 (D)$ of 
holomorphic functions with the Szeg\"o kernels, see \cite{Kran}, or their generalizations on Hilbert 
spaces of solutions to elliptic systems \cite{BerSch} for some specific domains. 

In any case, the applications of Reproducing Kernel Hilbert Spaces to optimization problems have some 
essential limits and hence, further advances in the kernel's method were realized in Reproducing Kernel Banach 
Spaces  or even in locally convex topological vector spaces, for instance, \cite{Bart_et_all, BreCar}. 
For these type of spaces the method is more effective if the related space is reflexive 
and  the Convex Analysis  on compact sets plays an essential role in the construction of a sparse minimizer for 
convex coercive lower semi-continuous functionals.  Of course, these advances are less constructive than in Hilbert 
spaces. However 
\cite[Theorem 4.14]{BreCar} produces a nice example of the sparse minimizing  over spaces 
of distributions of a finite order of singularity on a bounded open subset $\Omega$ of the 
$n$-dimensional real space ${\mathbb R}^n$, written as a linear combination of the values 
$\Phi (x^{(j)},\cdot)$  at distinguished points $x^{(1)}, \dots x^{(N)}$ 
of the training set with a fundamental solution $\Phi$ of a  
scalar differential operator $A$ with constant coefficients 
(note that the Theory of Radon measures play the crucial role in the 
exposition). The only disadvantage of this answer is that the minimizer 
was written up to  a solution of finite order of growth to the operator $A$ in $\Omega$.

Inspired by this example, we want to draw attention to a class of 
locally convex topological reproducing spaces $S_A (D) \cap {\mathcal B} (D)$ consisting 
of  solutions to the homogeneous equation 
$$
Au = 0 \mbox{ in } D 
$$
with an elliptic 
$(k\times k)$-matrix differential operator $A$ of order $m \in {\mathbb N}$ 
with coefficients defined on an open set $X \subset {\mathbb R}^n$, 
where $n\geq 2$ and  $D$ is a relatively compact domain in $X$; here the topology is induced 
from a  locally convex topological vector space ${\mathcal B} (D)$ of functions over $D$.

The most typical such a space 
is the Fr\'echet Montel vector space $S_A (D)$, endowed 
with the topology of uniform convergence on compact subsets of $D$.  If both the operators $A$ and 
$A^*$ possess the Unique Continuation Property on $X$ then
the famous Grothendieck duality,  see \cite{Grot1}, provides that the space 
$S_A (D)$ is  Montel Fr\'echet space with the reproducing 
kernel given by a suitable fundamental solution $\Phi$ to the operator $A$, see Proposition 
\ref{p.Montel.1} and Corollary \ref{c.repr.ker} below.

Another typical object is the space $S^{s,p}_A (D)$  of Sobolev solutions 
to the operator $A$ in $D$, i.e. the intersection of $S_A (D)$ with the 
Sobolev space $[W^{s,p} (D)]^k$. This  is a Banach space with reproducing kernel 
for all $s \in \mathbb Z$, $1<p<+\infty$, and we 
obtain the related Grothendieck type dualities in order to identify the kernel as 
a suitable fundamental solution $\Phi$, see Proposition 
\ref{p.Montel.1} and Corollary \ref{c.repr.ker}. 

On this way, the spaces $S_A (D)$ and $S^{s,p}_A (D)$ both fit  to the abstract schemes realized in 
\cite[\S 3]{Bart_et_all} and \cite[\S\S 2 and 3]{BreCar} for sparse minimizing 
(see also examples in \cite[\S 3]{LZZ}. However, in contrast with \cite{Bart_et_all, BreCar}, we prefer 
to work with spaces adapted to the investigation of boundary problem for PDE's. The Sobolev spaces 
$S_A^{s,p} (D)$  are 1) complete, reflexive (Theorem \ref{t.Groth.sp} below) and, at least for 
$s\in {\mathbb Z}_+$, $1<p<+\infty$, they are  uniformly convex,  2) they are used often in the theory 
of boundary value problems and 3) the elliptic boundary and volume potentials are usually acting  
continuously on them (that is important for Grothendieck type dualities). Additionally, for $s=0$, 
$ 1<p<+\infty$, the spaces, (and hence, their duals) are both  uniformly convex and uniformly smooth, 
see  \cite[P. 3, Ch II, \S 1, Proposition 8]{Bz} or  
\cite[Ex. 16.201]{NarBe}. Thus we choose the family of Banach spaces 
$S^{s,p}_A (D)$ to investigate the following minimization problem. 

\begin{prob} \label{pr.A.s.p.D}
Given $s \in \mathbb Z$, $p \in (1,+\infty)$, a linear continuous operator  ${\mathcal M} = {\mathcal M}_N$ 
that maps $S^{s,p}_A (D)$ onto ${\mathbb R}^{N }$ with some $N\in \mathbb N$, and a proper convex 
coercive lower semi-continuous functional $F = F_N : {\mathbb R}^{N } \to (-\infty, +\infty]$,  
find an element $u^{(0)} = u^{(0)} _N \in S^{s,p}_A (D)$ 
such that ${\mathcal F} (u^{(0)})  = \min_{u \in S^{s,p}_A (D) } {\mathcal F} (u)$, where
\begin{equation} \label{eq.F.s.p.D}
{\mathcal F} (u) = {\mathcal F}_N (u) = 
 F ({\mathcal M} u) + b_N \|u \|_{[W^{s,p} (D)]^k}  
\end{equation}
with a positive number $b_N$.
\end{prob}

Please, note that the space $S^{s,p}_A (D)$ is rather large because it is infinite dimensional for $n\geq 2$ 
regardless of the elliptic differential operator $A$. For  instance, for the Cauchy-Riemann operator 
$A=\overline \partial$ over the complex plane ${\mathbb C} \cong {\mathbb R}^2$ we have $k=2$ and the space 
$S^{s,p}_A (D)$ consists of holomorphic functions with some restrictions on growth in the domain $D$ and contains 
at least all the polynomials of the variable $z \in {\mathbb C}$ (but not of the variable 
$\overline z$!). Of course, for the scalar operator $\Delta^r$, $r \in \mathbb N$, (i.e. for 
the powers $\Delta^r$ of the Laplacian operator $\Delta$ in ${\mathbb R}^n$, $n \geq 2$) the space 
$S^{s,p}_A (D)$ consists of polyharmonic functions with certain restrictions on growth 
in the domain $D$ and contains at least all the polynomials of degree $(2r-1)$ of variables 
$(x_1, \dots x_n) \in {\mathbb R}^n$. Spaces of solutions of other types can be provided by the choice 
of the differential operator $A$ (polynomial-exponential solutions, etc.)

Taking into the account Representer Theorems for the Hilbert, Banach and topological spaces from 
\cite{Bart_et_all, BreCar, LZZ} , our 
goals are to prove that, under reasonable assumptions, 

1) there is a  minimizer $u^{(0)}= u^{(0)}_N$ to Problem \ref{pr.A.s.p.D}; 

2) there is a sparse minimizer $u^{\sharp} = u^{\sharp}_{N}$ to Problem \ref{pr.A.s.p.D} 
such that 
\begin{equation} \label{eq.sparse.u}
u^{\sharp} (x)= \sum_{j=1}^{N} \Phi(x,y^{(j)}) \, \vec {C} ^{(j)}
\end{equation} 
with some vectors $\vec {C} ^{(1)}, \dots \vec {C} ^{(N)}$ from ${\mathbb R}^k$ 
and points $ y^{(1)},  \dots y^{(N)}$ from ${\mathbb R}^n \setminus D$ 
where $\Phi (x,y)$ is the Schwartz kernel of the fundamental solution $\Phi$;

3) the sequence $\{ u^{(0)}_N\}$ of minimizers converges weakly to an 'approximate' 
solution to operator equation  \eqref{eq.ill-posed} as $N\to +\infty$. 

The paper is organized as follows. Section \S \ref{s.Grothendieck} is devoted to an adaptation of 
Grothendieck type dualities for spaces of Sobolev solutions to elliptic operators possessing the 
so-called Unique Continuation Property. In Section \S  \ref{s.sparse}  we prove the existence of the 
sparse minimizer $u^{\sharp}_{N}$ to Problem \ref{pr.A.s.p.D} in the form \eqref{eq.sparse.u} 
in the case of a smooth domain $D$ with connected complement and an elliptic differential 
operator $A$ such that both $A$ and its formal adjoint $A^*$ possess the Unique 
Continuation Property. Also, in this section we briefly look at the case where 
the number of data passes to infinity.  In Section \S \ref{s.appl} we discuss some applications 
to  problem related to neural networks and to the ill-posed Cauchy problem 
for elliptic operators  and to the well-posed Dirichlet problem for strongly elliptic operators. Some 
typical examples can be found there, too.

\section{Grothendieck type dualities}
\label{s.Grothendieck}
 
\subsection{Differential operators and function spaces}
\label{s.DO} 

In this  section we shortly recall the notion of differential operators, dualities and related matters. 

Let ${\mathbb R}^n$ be the Euclidean space with the coordinates $x= (x_1, \dots, x_n)$ endowed with 
the inner product $( x,y)_{{\mathbb R}^n} = y^T x  = \sum_{j=1}^n y_j x_j $ where $y^T$ denotes the transposed vector for a vector $y$. As usual, for $q \in [1, +\infty]$, let ${\mathbb R}^{n}_{ q}$ 
be  ${\mathbb R}^{n}$ endowed with the norm 
$$
\|x \|_{{\mathbb R}^{n}_{q } } = 
\left\{
\begin{array}{lll} 
\Big( \sum_{\nu=1}^n  |x_{\nu}|^q\Big)^{1/q}, & {\rm if} & 1\leq  q <+\infty, \\
\max_{1\leq \nu \leq n}  |x_{\nu}|, & {\rm if } &  q = +\infty. 
\end{array}
\right.
$$
As it is  known, all the norms $\|\cdot \|_{{\mathbb R}^{n}_{q } } $ are equivalent on 
${\mathbb R}^{n}$, for the dual space we have $({\mathbb R}^{n}_{q })^* \cong {\mathbb R}^{n}_{q' }$ 
with $1/q+1/q' =1$ and the related pairing is given by $\langle x,y\rangle = y^T x$.
 
We denote by $C(\sigma)$ the space of continuous functions on the set $\sigma \subset {\mathbb R}^n$.
If $\sigma$ is a compact set then we treat $C(\sigma)$ as a Banach space with the standard norm; 
for an open set $U$ we endow the space $C(U)$ with the standard Fr\'echet  topology of the uniform 
convergence on compact subsets of $U$. Then let $C^s ( U)$ be the standard Fr\'echet space of $s$ times 
continuously differentiable functions over an open set $U$ and let $C^s (\overline U)$ be the standard 
Banach space of $s$ times continuously differentiable functions over  $\overline U $ if $U$ relatively 
compact in ${\mathbb R}^n$.  As usual, $C^\infty_0 (U)$ stands for the space of infinitely 
differentiable functions with 
compact supports in the open set $U$,  ${\mathcal D}' (U)$ denotes the space of distributions 
over $U$, and ${\mathcal E}' (U)$ denotes the space of  distributions compactly supported in $U$. 

Next, let $L^p (\sigma)$, $1\leq p  <+\infty$ be the Lebesgue space on the measurable set $\sigma \subset
{\mathbb R}^n$ and let $W^{s,p} (U)$, $s \in {\mathbb Z}$, be the Sobolev spaces over open set 
$U\subset {\mathbb R}^n$. We endow the spaces with the standard norms $\| \cdot \|_{L^p (\sigma)}$, 
$\| \cdot \|_{W^{s,p} (U)}$ for  $s \in {\mathbb N}$, and 
$$
\| v\|_{W^{-s,p} (U)} = 
 \sup_{\|\varphi\|_{W^{s,p} (U)}\leq 1 \atop \varphi \in C^\infty_0 (U)} \Big|
\int_U \varphi(x) v(x) dx \Big| , \,\, s \in {\mathbb N};
$$
so, these are Banach spaces, see \cite{Adams}. As usual, for $s\in \mathbb N$, we denote 
by $W^{s,p}_0 (U)$ the closure of $C^\infty_0 (U)$ in $W^{s,p} (U)$. Thus, $W^{-s,p} (U)$ are dual 
spaces for $W^{s,p}_0 (U)$, $s \in \mathbb N$. We also use the so-called Besov spaces $B^{s,p} (U)$ 
with non-integer $s \in {\mathbb R}\setminus {\mathbb Z}$ that are Banach spaces 
with the standard norms, see \cite{BeINi}, \cite[\S 2.3]{Tri}; actually these are generalization 
of the Sobolev spaces $W^{s,p} (U)$ of integer $s$. 

Given $k \in \mathbb N$, $k\geq 2$, we systematically use the Cartesian product $[{\mathfrak C}]^k$ 
of toloplogical vector space ${\mathfrak C}$ with the component-wise topology; the elements of the 
space $[{\mathfrak C}]^k$ we treat as $k$-columns. 

Let $X$ be an open connected set in ${\mathbb R}^n$, $n\geq 2$. For integers $l,k$ we write 
$\mathrm{Diff}_{m} (X; {\mathbb R}^k \to {\mathbb R}^{l})$ for the space of all the  
$(l\times k)$-matrix 
(linear) partial differential operators of order $\leq m\in {\mathbb Z}_+$ over $X$, i.e. 
$A \in \mathrm{Diff}_{m} (X; {\mathbb R}^k \to {\mathbb R}^l)$ if 
\begin{equation*} 
A=\sum_{|\alpha|\leq m}a_\alpha(x) \partial^\alpha
\end{equation*}
where $a_\alpha(x)$ are $(l \times k)$-matrices of {\it real } $C^\infty (X)$-functions and 
$\partial^\alpha $ are partial derivatives of order $\alpha \in {\mathbb Z}_+^n$. We denote  
by $A^{\ast} \in \mathrm{Diff}_{m} (X; {\mathbb R}^l \to {\mathbb R}^{k})$ the corresponding 
formal adjoint operator
\begin{equation*} 
A^*=\sum_{|\alpha|\leq m} (-1)^{\alpha}  \partial^\alpha \big( a^*_\alpha(x) \cdot\big), 
\end{equation*}
where $a^*_\alpha$ is the adjoint matrix for $a_\alpha$ (actually, transposed, because the components 
of $a_\alpha$ are real in our particular case). 

We write $\sigma  (A)$ for the principal homogeneous symbol of the order $m$ of the operator $A$:
$$
\sigma(A) (x,\zeta) 
=\sum_{|\alpha|= m} a_\alpha(x) (\iota \zeta)^\alpha, \, x\in X, \, \zeta \in {\mathbb R}^n,
$$ 
where $\iota$ is the imaginary unit. 
We recall that $A$ is called (Petrovskii) elliptic on $X$ if $l=k$ and the mapping 
$\sigma  (A) (x,\zeta): {\mathbb C}^k  \to {\mathbb C}^k $ is invertible for all 
$(x,\zeta) \in (X, {\mathbb R}^k) $ with $\zeta\ne 0$, see, for instance, 
\cite[Ch 1, \S 3, Ch. 2, \S 2]{EgShu}. Since $\sigma(A^*) = \sigma^*(A)$, 
then $A$ is elliptic if and only if $A^*$ is elliptic.  

Let $D$ be a bounded domain (i.e. open connected set) in $X$.  Denote by $S_A (D)$ the space 
of all generalized solutions to the operator equation $Au=0$ in $D$. By the elliptic regularity,
$S_A (D) \subset [C^\infty (D)]^k$. Moreover, a priori estimates for solutions to elliptic systems 
imply that $S_A (D)$, when endowed with the topology of the uniform convergence 
on compact subsets in $D$, is a closed subspace of the Fr\'echet space $[C(D)]^k$, i.e.
a locally convex complete metrizable space itself, \cite{Shaef}). 

Now we need a specific property of solutions from the space $S_A (D)$.

\begin{defn} \label{def.US}
One says that $A$ possesses the  Unique Continuation Property (UCP) on $X$ if for any domain 
$D\subset X$ any solution $u \in S_A (D)$, vanishing on a non-empty open set $O \subset D$, is 
identically zero in $D$.
\end{defn}

Everywhere below, we assume that $A$ is elliptic on $X$ and  that both $A$ and $A^*$ possess the 
 Unique Continuation Property on $X$. Under these assumptions, the operator $A$ admits a bilateral 
fundamental solution $\Phi$ on $X$, \cite[\S 2.3]{Tark35}, i.e. for the Schwartz kernel $\Phi (x,y)$ 
of $\Phi$ we have 
\begin{equation}\label{eq.bilateral}
A_x \Phi (x, y) = \delta_{x-y}, \,\, A^*_y \Phi^{\textcolor{red}{*}} (x, y) = \delta_{x-y},
\end{equation}
where $\delta_0$ is the Dirac functional at the origin. In particular, 
$\Phi (x,y)$ is a $(k\times k)$-matrix of $C^\infty ((X\times X)\setminus \{x=y\})$ entries.
As it is known, UCP holds true for any elliptic operator 
with real analytic coefficients. Also scalar second order operators 
with $C^1$ smooth coefficients  possess the UCP, 
see, for instance, \cite{Lan56}. 

Next, we recall that a  Green operator
$G_A (\cdot, \cdot)$ for $A$ on $X$ is a bi-differential operator of order $(m-1)$ on $X$ mapping 
$[C^{m} (X)]^l \times[C^{m} (X)]^k $ to the space  $C^{1} (X,\Lambda^1)$, consisting on exterior $(n-1)$-
differential forms, and satisfying
$$
d \, G_A (g, u) = \Big( g^*  A u - (A^* g)^{*} u  \Big) dx \mbox{ on } X 
\mbox{ for all }   u \in [C^{m} (X)]^k,  g \in [C^{m} (X)]^l .
$$
A Green operator always exists, see \cite{Tark35}, and combined with the Stokes' formula results in the (first) Green formula:
\begin{equation} \label{eq.Green.1}
\int_{\partial D} G_A (g, u) = 
\int_{D} \Big( g^*  A u - (A^* g)^{*} u  \Big) dx
\end{equation}
for any relatively compact Lipschitz  domain $D$ in $X$ and  all 
$g \in [C^{m} (\overline D)]^l$, $u \in [C^{m} (\overline D)]^k $.  
Next, using a (left) fundamental solution $\Phi$ we immediately get the (second) Green formula for any 
solution $u \in S_A(D) \cap [C^{m} (\overline D)]^k$: 
\begin{equation} \label{eq.Green.2}
-\int_{\partial D} G_A (\Phi (x, \cdot), u) = \left\{ \begin{array}{lll} u(x), & x \in D, \\
0, & x \not \in \overline D.
\end{array}
\right. 
\end{equation}

In order to write the Green formulae in  more familiar form we proceed with the following definition.

\begin{defn} \label{def.Dir}
A set of linear differential 
operators $\{B_0,B_1, \dots B_{m-1}\}$ is called a 
$(k\times k)$-Dirichlet system of order $(m-1)$ 
on $\partial D$ if: 
1) the operators are defined in a neighbourhood of $\partial D$; 
2) the order of the differential operator $B_j $ equals to $j$; 
3) the map $ \sigma (B_j) (x,\nu (x)) :{\mathbb C}^k \to {\mathbb C}^k$ 
is surjective for each $x \in \partial D$, where 
$\nu (x)$  denotes the outward normal vector to the hypersurface $\partial D$
at the point $x\in \partial D$.
\end{defn}

\begin{lem} \label{eq.dual.Dir}
Let $m\in \mathbb N$, 
$\partial D\in C^m$,  $A$ be an elliptic differential operator 
of order $m \in \mathbb N$ in a neighbourhood of 
$\overline D$ and $B=\{B_0,B_1, \dots B_{m-1}\}$ be a Dirichlet system of order $(m-1)$ on $
\partial D$. Then there is a  Dirichlet system ${C }^A=\{ C_0^A, 
C_1 ^A,\dots  C_{m-1}^A \}$ 
with bijective symbols $ \sigma (C^{A}_j) (x,\nu (x)) :{\mathbb C}^k \to {\mathbb C}^k$ 
 for each $x \in \partial D$  such that 
for all $u,g \in [C^{m} (\overline D)]^k$ we have  
\begin{equation} \label{eq.Green.M.B}
\int_{\partial D} \Big( \sum_{j=0}^{m-1}({C}^A_{m-1-j} g)^* B_j u 
\Big) d\sigma = \int_{D} \Big( g^*  A u - (A^* g)^{*} u  \Big) dx.
\end{equation}

\end{lem}

\begin{proof} See, for instance, \cite[Lemma 8.3.3]{Tark37}.
\end{proof}

A typical $(k\times k) $-Dirichlet system of order $(m-1)$ on a  boundary $\partial D $ of class $C^m$ 
consists of $\{I_k, I_k \frac{\partial}{\partial\nu}, \dots, I_k \frac{\partial^{m-1}}{\partial\nu^{m-1}} \}$
where $I_k$ is the unit $(k\times k)$-matrix and $\frac{\partial^j}{\partial\nu^j} $ 
is $j$-th normal derivative with respect to $\partial D$. 

Next, for a non-open set $K\subset X$ denote by $S_A (K)$ the set of all solutions $u$ in neighbourhoods 
of $K$ in $X$, i.e. $u \in S_A (K)$ if and only if there is an open set $U_u \supset  K$ such 
that $u \in S_A (U_u)$.  As usual, we endow this space with the inductive limit topology 
of the family of spaces $\{ S_A (U_\nu)\}_{\nu \in {\mathbb N}}$ related to any sequence of domains  
$U_{\nu+1}\Subset U_{\nu} $, $\nu \in {\mathbb N}$, containing $ K$. 
Under the Unique Continuation Property, $S_A ( K)$ is a Hausdorff $DF$-space.
Actually, it can be treated as an $LB$-space, i.e. as an inductive limit of Banach spaces 
$\{ S_A (U_\nu)\cap C(\overline U_\nu)\}_{\nu \in {\mathbb N}}$, see \cite[Ch. 2, \S 6.3]{Shaef}. 

Next, given $s \in \mathbb Z$, $1< p < + \infty$ denote by $S^{s,p}_A (D) $ 
the intersection $ S_A (D) \cap [W^{s,p} (D)]^k$, i.e. the space 
of $W^{s,p}(D)$-Sobolev solutions to the operator $A$ in $D$. 
It is well-known to be   a Banach space. 

To extend the second Green formula for elements of $S^{s,p}_A (D) $  we invoke the 
notions of traces and weak boundary values. The Trace Theorem for the Sobolev spaces, see, for instance 
\cite[Ch.~1, \S~8]{LiMa72}, \cite{BeINi} 
implies that  
if $\partial D\in C^{s}$, $s\geq j+1\geq 1$,  then 
each operator $B_j$, $0\leq j \leq m-1$, 
induces a bounded linear operator
\begin{equation} \label{eq.trace.B_j}
B_j: [W^{s,p} (D)]^k \to [B^{s-j -1/p,p} (\partial D)]^k, \,\, 1 < p < +\infty.
\end{equation}

As it is known, for the spaces of negative smoothness
we may pass to the so-called weak boundary values, see, for instance, 
\cite[\S\S 9.3, 9.4]{Tark36} or \cite{Roit96} or \cite[Definition 2.2]{ShTaLMS} 
for general elliptic systems satisfying the Unique Continuation Property. Namely,  
if $\partial D\in C^\infty$ then each operator $B_j$, $0\leq j \leq m-1$, 
induces a bounded linear operator, see \cite[Chapters 6, 10]{Roit96}, 
\begin{equation} \label{eq.traces.int}
B_j: [S_A^{-s,p} (D)]^k \to 
[B^{-s-j-1/p,p} (\partial D)], \,
s\in {\mathbb Z}_+, \, 1 < p < +\infty.
\end{equation}  
Now we introduce the linear operator ${\mathcal G}: 
\oplus_{j=0}^{m-1} [B^{s-j-1/p,p} (\partial D)]^k \to S_A^{s,p} (D) $ defined by 
\begin{equation}\label{eq.G}
{\mathcal G} \big( \oplus_{j=0}^{m-1} u_j\big) (x) = 
\left\{
\begin{array}{lll}
- \int_{\partial D} \Big( \sum_{j=0}^{m-1}({C}^A_{m-1-j} \Phi^* (x, \cdot))^* 
u_j  \Big) d\sigma, & m\leq s \in \mathbb N,\\
- \sum_{j=0}^{m-1} \langle  {C}^A_{m-1-j} \Phi^* (x, \cdot)
 , u_j \rangle_{\partial D}, & m>s \in \mathbb Z,\\
\end{array}
\right.
\end{equation}
(here $\langle \cdot, \cdot \rangle_{\partial D}$ stand for the pairing between the space 
$B^{s-j-1/p,p} (\partial D) $ and its dual $B^{j+1/p-s,p'} (\partial D) $); according to \cite[\S 2.3.2.5]{RS82}, 
\cite[\S 2.4]{Tark36}, 
it is bounded.

\begin{thm} \label{t.Green}
Let $D$ be a relatively compact domain in $X$ with $C^\infty$-smooth boundary and $A$ be an elliptic 
operator on $X$  with UCP. Then the (second) Green formula is valid for any 
$u \in S_A ^{s,p} (D)$, $s\in \mathbb Z$, $1<p<+\infty$:
\begin{equation} \label{eq.Green.M.B.2}
-{\mathcal G} \big( \oplus_{j=0}^{m-1} B_j u \big) = 
\left\{ \begin{array}{lll} u(x), & x \in D,\\
0, & x \not \in \overline D.
\end{array}
\right. 
\end{equation}
\end{thm}

\begin{proof} See \cite[\S 9.4]{Tark36}, \cite[Theorem 2.4]{ShTaLMS}.
\end{proof}

Actually, we may characterize the  boundary values (both weak values and  traces) of elements  from 
$S_A ^{s,p} (D)$. 

\begin{lem} \label{l.chract.traces.A}
Let $s\in \mathbb Z$, $1<p<+\infty$. 
For a set $\oplus _{j=0}^{m-1 }u_j \subset 
\oplus _{j=0}^{m-1 } B^{s-j-1/p,p} (\partial D) $ there is an element $u \in S_A ^{s,p} (D)$
satisfying $\oplus _{j=0}^{m-1 }u_j = \oplus _{j=0}^{m-1 } B_ju$  in the sense 
of weak boundary values on $\partial D$ if and only if 
\begin{equation} \label{eq.chract.traces.A}
{\mathcal G} (\oplus_{j=0}^{m-1} u_j) (x)=0 \mbox{ for all } x \not \in \overline D. 
\end{equation}
\end{lem}

\begin{proof} Cf. also \cite[\S 10.3.4]{Tark36}. 
Indeed, the necessity of \eqref{eq.chract.traces.A} follows immediately from 
formula \eqref{eq.Green.M.B.2}. Back, if \eqref{eq.chract.traces.A} holds true then 
we invoke the Theorem on weak jumps of Green type integrals, see \cite[Lemma 2.7]{ShTaLMS} or 
\cite[\S 10.1.2]{Tark36}:
\begin{equation} \label{eq.weak.jump}
\langle g_i, B_i ({\mathcal G}(\oplus_{j=0}^{m-1} u_j))^-_{|\partial D} -
B_i ({\mathcal G}(\oplus_{j=0}^{m-1} u_j))^+_{|\partial D} \rangle _{\partial D} = 
\langle g_i,u_i \rangle_{\partial D} 
\end{equation}
for all $ g_i \in  [C^{\infty} (\partial D)]^k$, 
where ${\mathcal G}(\oplus_{j=0}^{m-1} u_j))^-$ stands for the restriction of 
${\mathcal G}(\oplus_{j=0}^{m-1} u_j)$ to $D$ and
${\mathcal G}(\oplus_{j=0}^{m-1} u_j))^+$ is the restriction of 
${\mathcal G}(\oplus_{j=0}^{m-1} u_j)$ to $X\setminus \overline D$. Thus, 
combining \eqref{eq.chract.traces.A} and \eqref{eq.weak.jump} we see that 
$$
B_i {\mathcal G} (\oplus_{j=0}^{m-1} u_j)^{-}_{|\partial D} = u_i, \, 0\leq i \leq m-1,
$$
that was to be proved.
\end{proof}

\begin{rem} \label{r.norm.A}
In particular,  by Green formula \eqref{eq.Green.M.B.2} and the continuity 
of Green integrals \eqref{eq.G} and boundary operators \eqref{eq.traces.int}, we see that 
 the functional
\begin{equation} \label{eq.norm.A}
\|u \|_{s,p,B} = 
\Big( \sum_{j=0}^{m-1 }\|B_{j} u \|^p_{B^{s-j-1/p,p} (\partial D)} \Big)^{1/p}
\end{equation}
defines a norm on the space $S_A ^{s,p} (D)$, equivalent to the original one. 
\end{rem}

Now we recall that a solution $u \in S_A(D)$ has a finite order of growth near $\partial D$, 
if for each point $x^{(0)} \in \partial D$ there are positive numbers $\gamma$, $C$ and $R$ such that 
$$
|u (x) | \leq C |x-x^{(0)}|^{-\gamma} \mbox{ for all } x \in D, |x-x^{(0)}|<R. 
$$
The space of such functions we denote by $S_A^{(F)} (D)$. 

\begin{cor} \label{c.Green}
Let $D$ be a relatively compact domain in $X$ with $C^\infty$-smooth boundary and $A$ be an elliptic 
operator on $X$  with UCP. Then  
\begin{equation}\label{eq.SF}
S_A ^{(F)} (D)  = 
\cup_{s\in {\mathbb Z}} \, S_A^{s,2} (D) = \cup_{s\in {\mathbb Z} \atop 1<p<+\infty} \, S_A^{s,p} (D)
\end{equation} 
and the (second) Green formula \eqref{eq.Green.M.B.2} is still valid for any 
$u \in S_A ^{(F)} (D)$. 
\end{cor}
 
\begin{proof} 
Indeed, according to 
\cite[Chapter 9]{Tark36} (or \cite[Theorem 2.6]{ShTaLMS}), 
a vector $u$ belongs to $S_A ^{(F)} (D)$ if and only if $u$ admits distributional 
weak boundary values $u_j$ for $B_j u$, $0\leq j \leq m-1$. Moreover, in this case $u$ can be presented by 
Green formula \eqref{eq.Green.M.B.2}. Then the statement follows because $\partial D$ is compact and hence,  
any distribution on it has a finite order of singularity, i.e. 
for each $u \in S_A^{(F)} (D)$ there is $s\in \mathbb Z$ such that 
$\oplus _{j=0}^{m-1}B_j u \in \oplus _{j=0}^{m-1}B^{s-1/j-1/2,2} (\partial D)$.  
In particular, $S_A ^{(F)} (D)  = \cup_{s\in {\mathbb Z}} \, S_A^{s,2} (D)$.
The second identity in \eqref{eq.SF} follows from the Sobolev Embedding Theorems.
\end{proof}

Thus, we may endow the space $S_A ^{(F)} (D)$ with the inductive limit topology with respect to the family 
$\{S^{s,2}_A (D)\}_{s\in {\mathbb Z}}$ of Banach spaces (an $LB$-space), see, for instance, \cite[\S 6.3]{Shaef}. 
According to \cite[Ch. 4, Exercise 24e]{Shaef},  $S_A ^{(F)} (D)$ is a ${\rm DF}$-space. 

At the end of this section we discuss spaces with reproducing kernel.

\begin{defn}
A topological vector space ${\mathcal B} = {\mathcal B} (D)$ of $k$-vector functions over the domain $D$ 
is called a space with reproducing kernel if 
\begin{enumerate}
\item[(1)] 
the vector space $\mathcal B$  is endowed with the point-wise operations of sum and
multiplication by a scalar;
\item[(2)]
the evaluation operators ${\mathcal B} \ni u \to u(x) \in {\mathbb R}^k$ (functionals if $k=1$) are linear and continuous 
on ${\mathcal B}$ for each $x \in D$.
\end{enumerate}
\end{defn}
This means that for each $x \in D$ the evaluation functionals 
${\rm ev}^{(j)}_x u = u_j(x)$, $1\leq j \leq k$,  belong to  
topological dual space ${\mathcal B}^*$ and for the operator ${\rm ev}_x = 
({\rm ev}^{(1)}_x, \dots {\rm ev}^{(k)}_x)^T$  we have 
\begin{equation} \label{eq.RK.TVS}
u (x) = \langle {\rm ev}_x, u\rangle _{{\mathcal B}^*,{\mathcal B}} \mbox{ for all } u \in {\mathcal B},
\end{equation} 
where $\langle \cdot, \cdot \rangle _{{\mathcal B}^*,{\mathcal B}}$ is the pairing 
between ${\mathcal B}^*$ and ${\mathcal B}$. Thus, a reproducing kernel ${\mathcal K} (x ,\cdot)$ 
is presented essentially by the evaluation operator  ${\rm ev}_x$ via formula \eqref{eq.RK.TVS}.

We emphasize that  the operator ${\rm ev}_x $  depend essentially on the pairing 
$\langle \cdot, \cdot \rangle _{{\mathcal B}^*,{\mathcal B}}$. Of course, if 
${\mathcal B}$ is a (separable) Hilbert space then the Riesz theorem grants the canonical 
pairing given by the inner product $(\cdot,\cdot)_{{\mathcal B}}$; in this case 
the reproducing kernel ${\mathcal K} (x,y)$ representing the operator ${\rm ev}_x $  
may be defined as 
\begin{equation} \label{eq.RK.HS}
{\mathcal K} (x,y) = \sum_{\nu=1}^{\dim{{\mathcal B}}} b_\nu (x) \otimes b_\nu (y), \, x,y\in D.
\end{equation} 
with any orthonormal basis $\{ b_\nu\}_{\nu=1}^{\dim{{\mathcal B}}} $ in ${\mathcal B}$, \cite{Ar}. 
In particular, in this case we have
$$
u (x) = ({\mathcal K} (x,\cdot), u )_{{\mathcal B}} \mbox{ for all } u \in {\mathcal B} 
\mbox{ and all } x \in D.
$$
For non-Hilbert spaces there are no such a simple-looking formulae for the reproducing kernel and 
the action of the evaluation functional. But even for Hilbert spaces with reproducing kernel 
formula \eqref{eq.RK.HS} is rarely  written in an accomplished explicit form. 
The exceptions are the Szeg\"o and Bergman kernels for Hardy and Bergman spaces, respectively, of 
holomorphic functions in domains of a special kind, see \cite{Berg}, \cite{BerSch}, \cite{Kran}.
However, \cite{Bart_et_all} gives an example of a class of {\it integral} Reproducing Kernel Banach Spaces
where the theory of the Radon measure allows to produce some reasonably looking reproducing kernels.

To formulate the next statement we recall that 
a Montel space is a barrelled topological vector space where any bounded closed set is compact.

\begin{prop} \label{p.Montel.1} The spaces $S_A (D)$ and $S_A (D)\cap [C^\infty (\overline D)]^k$ are 
Montel Fr\'echet spaces with reproducing kernels, 
the spaces $S_A (\overline D)$, $S^{(F)}_A (D)$ are Montel $DF$-spaces with reproducing kernels, 
$ S_A ^{s,p} (D)$, $s \in \mathbb Z$, $1< p < + \infty$, are Banach spaces  
with reproducing kernel, $ S_A ^{s,2} (D)$, $s \in \mathbb Z$, are Hilbert spaces with 
reproducing kernels.
\end{prop} 

\begin{proof} Obviously, for each $x \in D$ and all $u \in S_A (D)$
$$
|u(x)|\leq \|u\|_{[C (K)]^k}
$$
for any compact set $K\subset D$ containing $x$. Hence, the Dirac functional 
$\delta_x u \to u(x)$ is linear and continuous on the Fr\'echet space $S_A (D)$ with the topology 
of the uniform convergence on compact subsets of $D$. Let us see that $S_A (D)$ is a Montel-Fr\'echet space. 
As it is known in metric spaces the (Heine--Borel) compactness 
and sequential compactness are equivalent, so we use the sequential compactness. By the elliptic regularity,
(one also may use the second Green formula \eqref{eq.Green.2}, with this purpose), 
we conclude that 
for each domains $D_1 \Subset D_2 \Subset D$ we have 
$$
\|u\|_{C^s (\overline D_1)} \leq C_s(D_1,D_2) \|u \|_{C (\overline D_2)}
$$
for all $u \in S_A (D)$. This means that $S_A (D)$, endowed with 
the topology of $[C(D)]^k$ is actually a closed subset in Fr\'echet space $C^\infty (D)$ endowed 
with the topology of convergence on compact subsets of $D$ with all the partial derivatives 
of arbitrary high orders. As it is known,  the  Fr\'echet  topology in $[C^\infty (D)]^k$  may be   
induced by the family of semi-norms  $\{ \|\cdot\|_{[C^\nu(\sigma_\nu)]^k}\}_{\nu \in \mathbb N}$ 
where $\sigma _\nu =\overline U_\nu $ for open sets $U_\nu \Subset U_{\nu+1}\Subset D$ with 
$\cup_{\nu \in \mathbb N} U_\nu = D$. A bounded set $M$ in $[C^\infty (D)]^k$  is a set bounded with respect 
to each semi-norm $ \|\cdot\|_{[C^\nu(\sigma_\nu)]^k}$. As the embeddings
$$
[C^{\nu+1}(\sigma_{\nu+1})]^k \to [C^\nu(\sigma_\nu)]^k
$$ 
are compact for all $\nu \in \mathbb N$
we conclude that $[C^\infty (D)]^k$ is a Montel-Fr\'echet space, see \cite[p. 75]{RobRob} for spaces 
of functions of one variable. If $M$ is a bounded closed set in $S_A (D)$ then it is bounded closed set in  
$[C^\infty (D)]^k$ and hence, it is a compact in $[C^\infty (D)]^k$. However, the limit point of $M$ again 
belong to $S_A (D)$ because space's topology.    Thus, $S_A (D)$ is a Montel space, \cite{Shaef}.

Similarly, 
for each $x \in D$ and all $u \in S_A (D) \cap [C^\infty (\overline D)]^k$ we have 
$$
|u(x)|\leq \|u\|_{[C (\overline D)]^k}
$$
Hence, the Dirac functional 
$\delta_x u \to u(x)$ is linear and continuous on the Fr\'echet space 
$S_A (D) \cap [C^\infty (\overline D)]^k$ with the topology induced by 
semi-norms $\{ \|\cdot\|_{[C^\nu(\overline D)]^k}\}_{\nu \in \mathbb N}$. 
 As the embeddings
$$
[C^{\nu+1}(\overline D)]^k \to [C^\nu(\overline D)]^k
$$ 
are compact for all $\nu \in \mathbb N$
we conclude that $[C^\infty (D)]^k$ is a Montel Fr\'echet space, 
see \cite[p. 75]{RobRob} for spaces 
of functions of one variable. Hence, we may conclude that $S_A (D) \cap [C^\infty (\overline D)]^k$ 
is a Montel-Fr\'echet space,  as for the space 
$S_A (D)$.
 
The space $S_A (\overline D)$ is a Montel space as the inductive limit of the Banach 
spaces  $\{ S_A (U_\nu ) \cap [C(\overline U_\nu)]^k\}$,  related to any sequence of open subsets  
$U_{\nu+1}\Subset U_{\nu} $, $\nu \in {\mathbb N}$, containing $ \overline D$ (\cite[Ex. 24, pp. 196-197]{Shaef}). 
The evaluation operator ${\rm ev}_x : S_A (\overline D) \to {\mathbb R}^k$ is continuous 
because it is continuous on each Banach space $ S_A (U_\nu )\cap [C(\overline U_\nu)]^k$, ${\nu \in \mathbb N}$, 
by the discussion above. Moreover, by a priori elliptic estimates the space $S_A (\overline D)$ may be alternatively 
considered as the inductive limit of the Banach spaces   
$\{ S^{s,p}_A (U_\nu ) \}$ with any fixed $s\in \mathbb Z$ and $1<p<+\infty$.

Next, by the continuity of the operators \eqref{eq.trace.B_j}, \eqref{eq.traces.int}
we conclude that there is a constant $C(x)>0$ such that 
\begin{equation} \label{eq.RK.sp}
|u(x)|\leq \sum_{j=1}^{m-1}\|B_j u\|_{[B^{s-j-1/p,p} (\partial D)]^k}
\|C_{m-j-1} \Phi ^*(x,\cdot)\|_{[B^{j+1/p-s,q} (\partial D)]^k} \leq
\end{equation} 
$$
 C(x) \|u\|_{[W^{s,p} (D)]^k}
$$
for all $u \in S_A^{s,p} (D)$. Therefore, $S_A^{s,p} (D)$ are Banach spaces with reproducing kernels.

The space $S_A^{(F)} (D)$ is a DF-space because it is an inductive limit of Banach spaces $S_A^{s,p} (D)$, 
\cite[Ex. 24, pp. 196-197]{Shaef}. The continuity of the evaluation operator on it follows from 
estimate \eqref{eq.RK.sp}. 
\end{proof}

In order to identify the reproducing kernels we need a theorem on a general form of 
continuous linear functionals on the spaces, specifying the pairing 
$\langle \cdot, \cdot \rangle _{{\mathcal B}^*,{\mathcal B}}$ in \eqref{eq.RK.TVS}. 
The (second) Green formula \eqref{eq.Green.2}, providing the reproducing property, suggests that the kernels 
should be connected somehow with the fundamental solution $\Phi (x,y)$ of the operator $A$. 

\subsection{The Grothendieck  type dualities 
for spaces of solutions to elliptic operators
}
\label{s.Groth} 

The type of dualities for spaces of solutions to elliptic operators, we are looking for,  
appeared first in papers by J. Silva \cite{Silva}, G. K{\"o}the \cite{Koth2} and A. Grothendieck \cite{Grot2} 
devoted to holomorphic functions. Then A. Grothendieck \cite{Grot1} obtained a generalization 
of these results for general elliptic systems admitting a bilateral fundamental solution. 
Following  \cite{Grot1} we say that a solution $u \in S_{A}(X\setminus \overline D)$ 
is 'regular at infinity'  with respect to a fundamental solution $\Phi$ for $A$ if 
\begin{equation} \label{eq.reg.inf}
\langle \Phi ^* g, A (\varphi u) \rangle _{X\setminus \overline D} =0
\end{equation}
for any vector $g \in [C^\infty _0 (X)]^k$ and 
any function $\varphi \in C^\infty _0(X)$ that equals to $1$ in a neighbourhood of $\overline D \cup 
{\rm supp} \,  g$. 
The set of elements from that are  'regular at infinity' 
will be denoted by $S_{A}(\hat X\setminus   \overline D)$; 
the space of solutions on the non-open set $X\setminus D$, 'regular at infinity' will be denoted 
by $S_{A}(\hat X\setminus  D)$;

Of course, this definition depends essentially on the fundamental solution $\Phi$.
According to \cite[Corollary 1, p.258]{Grot1},  $u \in S_{A}(\hat X\setminus  \overline D)$ 
if and only if $u$ coincides outside of a compact set $K$ containing $\overline D$ with 
a vector-function $v$ of type 
\begin{equation} \label{eq.Groth.reg.inf}
v(x) = 
\langle \Phi ^* 
(x,\cdot), \phi \rangle 
\end{equation}
 for a compactly supported 
vector-distribution $\phi$  in $X$ (i.e. $\phi \in [{\mathcal E}' (X)]^k)$.

Moreover, it follows from \cite[formula (18), p.258]{Grot1}, that  
 if $u \in S_{A}(\hat X\setminus  \overline D)$ then 
\begin{equation} \label{eq.Green.2.compl}
\int_{\partial \Omega} G_A (\Phi (x, \cdot), u) =  u(x), \,  x \in X \setminus \overline 
\Omega, 
\end{equation} 
for any domain $\Omega$ with 
a Lipschitz boundary satisfying $\overline D \subset \Omega\Subset X$,  see also \cite[formula (5.4.6)]{Tark35}. 

We endow the space 
$S_A (X\setminus \overline D)$  with the standard 
topology of the uniform convergence on compact subsets in $X\setminus \overline D$. 
Formula \eqref{eq.Green.2.compl} implies that it is a complete metric space
because any limit of elements of  this space in  Fr\'echet space 
$S_A (X\setminus \overline D)$  is again 'regular at infinity'. Thus, 
$S_{A}(\hat X\setminus   \overline D)$ is a Fr\'echet space, too. 

The space $S_{A}(\hat X\setminus  D)$ is endowed with the standard inductive limit topology 
related to the family $\{  S_{A}(\hat X\setminus  K_\nu)\}_{\nu \in {\mathbb N}} $ for a 
monotone increasing sequence of compact sets $K_\nu \subset D$. 

The Grothendieck's results may be formulated as follows. 

\begin{thm}[A. Grothendieck, \cite{Grot1}] 
\label{t.Groth.Fr} 
Let $A$ be an elliptic operator on $X$ with the UCP for both $A^*$ and $A$, 
$D$ be a relatively compact domain  in $X$. If  $X\setminus D$ has no compact components in $X$  
then the strong dual  $(S_A  (D))^*$ for the Fr\'echet space $S_A  (D)$ is topologically isomorphic to 
the DF-space $ S_{A^*} (\hat X \setminus  D)$. 
\end{thm}

More precisely, he proved that 

1) for any $ v \in S_{A^*} (\hat X \setminus D) $ the functional 
\begin{equation} \label{eq.pairing.Gr}
f_v (u) = - \int_{\partial \Omega} G_A (v, u)  , \, 
u \in S_{A} (D) ,
\end{equation}
is linear and continuous over $S_{A} (D)$ and it 
does not depend on any relatively compact domain $\Omega\subset D$ with Lipschitz boundary such that 
$\partial \Omega \subset {\rm Dom} (v) \cap D$;

2) for any  $f \in (S_A  (D))^*$ there is an element $v \in 
S_{A^*} (\hat X \setminus  D) $ such that $f(u) = f_v (u)$ for all  $u\in S_{A} (D)$;

3) the mapping $S_{A^*} (\hat X \setminus  D) \ni
v \to f_v \in (S_A  (D))^*$  is a homeomorphism of these spaces.

\begin{cor} \label{c.Groth.Fr}
Let $A$ be an elliptic operator on $X$ with the UCP for both $A^*$ and $A$, $D$ be a relatively compact domain 
in $X$. If  $X\setminus D$ has no compact components in $X$ then

1) the strong dual  $(S_{A}   (\overline D))^*$ for the  Montel DF-space $S_{A}  (
\overline D)$ is topologically isomorphic to 
the Fr\'echet  space $ S_{A^*} (\hat X\setminus \overline D)$;

2)  the strong dual $( S_{A^*} (\hat X \setminus  D))^*$ for the  Montel DF-space $ S_{A^*} (\hat X \setminus  D)$
is topologically isomorphic to the Montel Fr\'echet space $ S_{A} (D)$;

3) the strong dual  $(S_{A^*}  (\hat X\setminus \overline D))^*$ for the  Montel Fr\'echet space $S_{^*A}  (
\hat X\setminus \overline D)$ is topologically isomorphic to 
the  Montel DF-space $ S_{A} (\overline D)$. 
\end{cor}

\begin{proof} The proof of the first statement is similar to the proof of \cite[Theorem 4]{Grot1}.
The statements 2) and 3) follow because Montel spaces are reflexive and the dual to a Montel space 
is a Montel space, see \cite[\S 5.8]{Shaef}, \cite[p. 74]{RobRob}. Note that for the first duality 
the pairing is constructed as follows:
 for any $ v \in S_{A^*} (\hat X \setminus \overline D) $ the  functional 
\begin{equation} \label{eq.pairing.Gr.closure}
f_v (u) = - \int_{\partial \omega} G_A (v, u)  , \, 
u \in S_{A} (\overline D) ,
\end{equation}
does not depend on any relatively compact domain $\omega\supset D$ with Lipschitz boundary such that 
$\partial \omega \subset {\rm Dom} (u) \cap D$; it gives actually the general form of    
linear bounded functionals on $S_{A} (\overline D)$.
\end{proof}

We want to extend these facts to the Banach spaces  $S_A ^{s,p} (D)$, $s\in \mathbb Z$, 
$1<p<+\infty$. With this purpose, 
 we denote by $ \tilde S_{A^*}^{s',q} (\hat X \setminus \overline D) $ the set of vector functions 
from $ S_{A^*} (\hat X \setminus \overline D)$ that belong to $  [W^{s',q}
(\Omega \setminus \overline D)]^k$, $s'\in \mathbb Z$, $1<q<+\infty$, for any relatively compact domain
$\Omega \subset X$, containing $\overline D$. Taking into account the continuity of boundary operators
\begin{equation} \label{eq.trace.C_j}
C^A_j: [W^{s',q} (\Omega \setminus \overline D)]^k \to 
[B^{s'-j -1/q,q} (\partial D)]^k, \, s'>j.
\end{equation}
\begin{equation} \label{eq.traces.ext}
C^A_j: [S_{A^*}^{s',q} (\Omega \setminus \overline D)]^k \to  
[B^{s'-j-1/q,q} (\partial D)]^k, \, \,  s'\leq j,
\end{equation} 
we endow it  with the norm 
$$
\|v \|_{s',q,C^A} = 
\Big( \sum_{j=0}^{m-1 }\|C^A_{j} v \|^q_{B^{s' -j -1/q,q} (\partial D)}  \Big)^{1/q},
$$ 
where $\{ C_j^A\}_{j=0}^{m-1}$ is a Dirichlet system on $\partial D$ dual to the system 
 $\{ B_j\}_{j=0}^{m-1}$ with respect to (first) Green formula \eqref{eq.Green.M.B}.

Let us see that a natural (second) Green formula still holds for elements of the space
$\tilde S^{s',q}_{A^*} (\hat X \setminus \overline D) $. For this reason 
we introduce the linear operator 
\begin{equation}\label{eq.G*.act}
\tilde {\mathcal G}: 
\oplus_{j=0}^{m-1} 
[B^{s'-m+j+1-1/q,q} (\partial D)]^k \to \tilde S_{A^*}^{s',q} (\hat X\setminus \overline D) 
\end{equation}
defined 
by 
\begin{equation}\label{eq.G*}
\tilde {\mathcal G} \big( \oplus_{j=0}^{m-1} v_j\big) (y)= 
\left\{
\begin{array}{lll}
- \int_{\partial D} \Big( \sum_{j=0}^{m-1}  
 (B_j \Phi (\cdot,y))^* v_{j}  \Big) d\sigma, & m\leq s' \in \mathbb N,\\
- \sum_{j=0}^{m-1} \langle B_j \Phi (\cdot,y)   
,  v_{j}\rangle_{\partial D}, & s'<m \in \mathbb Z,\\
\end{array}
\right.
\end{equation}
(here $\langle \cdot, \cdot \rangle_{\partial D}$ stands for the pairing between the space 
$B^{s'-m+j+1-1/q,q} (\partial D) $ and its dual $B^{m+1/q-s'-j-1,p} (\partial D) $); according to \cite[\S 2.3.2.5]{RS82}, 
\cite[\S 2.4]{Tark36}, 
it induces linear bounded mappings 
$$
\tilde {\mathcal G}: 
\oplus_{j=0}^{m-1} 
[B^{s'-m+j+1-1/q,q} (\partial D)]^k \to S_{A^*} (\Omega\setminus \overline D) \cap 
[W^{s',q}
(\Omega \setminus \overline D)]^k,
$$  
for any smooth relatively  compact domain $\Omega \subset X$, containing $\overline D$. 
Moreover, by  \cite[Corollary 1, p.258]{Grot1} (see also \eqref{eq.Groth.reg.inf}) 
$\tilde {\mathcal G} \big( \oplus_{j=0}^{m-1} v_j\big)$ belongs to 
$\tilde S_{A^*}^{s',q} (\hat X\setminus \overline D)$. The continuity 
of the operators \eqref{eq.trace.C_j}, \eqref{eq.traces.ext} implies the continuity of operator 
\eqref{eq.G*.act}.

\begin{lem} \label{l.Green.A*} For each element $v \in 
\tilde S_{A^*}^{s',q} (\hat X \setminus \overline D)$, $s'\in {\mathbb Z}$, $1<q<+\infty$, 
the following (second) Green formula holds true:
\begin{equation} \label{eq.Green.M.C.2}
\tilde {\mathcal G} \big( \oplus_{j=0}^{m-1} C^A_{m-j-1} v \big) (y) = 
v(y) \mbox{ for all } y \in X\setminus \overline D.\\
\end{equation}
\end{lem}

\begin{proof} Applying \eqref{eq.Green.2.compl} we see
that  
 if $v \in \tilde S^{s',q}_{A^*}(\hat X\setminus  \overline D)$ then 
\begin{equation} \label{eq.Green.2.compl*}
\int_{\partial \Omega} G_{A^*} (\Phi (\cdot,y), v) =  v(y) 
\mbox{ for all }  y \in X \setminus \overline 
\Omega, 
\end{equation} 
for any domain $\Omega$ with 
a Lipschitz boundary satisfying $\overline D \subset \Omega\Subset X$,  
 see also \cite[formula (5.4.6)]{Tark35}. For $\varepsilon <0$, denote by
$D_\varepsilon$ the bounded domain containing $\overline D$ such that 
the distance from $\partial D_\varepsilon$ and $\partial D$ equals to $|\varepsilon|$.
If $\partial D \in C^\infty$ and $|\varepsilon|$ is sufficiently small 
then $D_\varepsilon$ is a smooth domain, too. 
Besides, if $D_\varepsilon \Subset \Omega$ then
\begin{equation} \label{eq.eps}
\int_{\partial \Omega} G_{A^*} (\Phi (\cdot,y), v) = 
\int_{\partial D_\varepsilon} G_{A^*} (\Phi (\cdot,y), v) \mbox{ for all  } y \in X \setminus \overline \Omega, 
\end{equation}
because of the Stokes formula and the relations
$A_x \Phi (x,y) = 0 $ in $ D$ for $y \in X \setminus \overline \Omega$.

On the other hand, taking into account the relation between  $G_{A^*}$ and $G_{A}$  
(\cite[Proposition 2.4.5]{Tark35}) we see that 
$$
\int_{\partial D_\varepsilon} G_{A^*} (\Phi^* (\cdot,y), v) =  - 
\int_{\partial D_\varepsilon} G_{A} (v, \Phi^* (\cdot,y)) = 
$$
$$- 
\int_{\partial D_\varepsilon} \Big( \sum_{j=0}^{m-1}\big(B_j \Phi (\cdot,y)\big)^*  {C}^A_{m-1-j} v   \Big) d\sigma
\mbox{ for all  } y \in X \setminus \overline \Omega.
$$
In particular, passing to the limit with respect to $\varepsilon \to -0$ in \eqref{eq.eps} and 
using the existence of traces/boundary values of $v$ on $\partial D$ we obtain \eqref{eq.Green.M.C.2} 
for all $y \in X \setminus \overline \Omega$ for any domain $\Omega$ with 
a Lipschitz boundary satisfying $\overline D \subset \Omega\Subset X$, i.e. 
for all $y \in X\setminus \overline D$.  
\end{proof}

It follows from \eqref{eq.Green.M.C.2} and the continuity of operator \eqref{eq.G*.act} 
that $\tilde S^{s',q}_{A^*} (\hat X \setminus \overline D) $, $s'\in {\mathbb Z}$, 
$1<q<+\infty$, is a Banach space. 

Now we are to prove the main result of this section.

\begin{thm} \label{t.Groth.sp}
Let $A$ be an elliptic operator of order $m$ on $X$ with the UCP for both $A^*$ and $A$, 
$D$ be a relatively compact domain with $C^\infty$-smooth boundary in $X$, and 
$\{ B_j\}_{j=0}^{m-1}$ is a Dirichlet system on $\partial D$. If  
$X\setminus D$ has no compact components in $X$ 
then the strong dual  $(S_A ^{s,p} (D))^*$ for the space $S_A ^{s,p} (D)$ is topologically isomorphic to 
the Banach space  $ \tilde S_{A^*}^{m-s,p'} (\hat X \setminus \overline D) $, where $\frac{1}{p} + 
\frac{1}{p'} = 1$.
\end{thm}

\begin{proof} More precisely, we have to prove that 

1) for any $ v \in \tilde S_{A^*}^{m-s,p'} (\hat X \setminus \overline D) $ the functional 
induced by the pairing
\begin{equation} \label{eq.pairing.Gr.fin}
\langle v, u\rangle _{{\rm Gr}}= -\sum_{j=0}^{m-1} \langle {C}^A_{m-1-j} v,  
B_j u  \rangle_{\partial D}, \, 
u \in S_{A}^{s,p} (D) ,
\end{equation}
is linear and bounded over $S_{A}^{s,p} (D)$ (here $\langle \cdot, \cdot \rangle_{\partial D}$ stands for the pairing between 
$B^{s-j-1/p,p} (\partial D) $ and its dual $B^{j+1/p-s,p'} (\partial D) $);

2) for any  $f \in (S_A ^{s,p} (D))^*$ there is an element $v \in 
\tilde S_{A^*}^{m-s,p'} (\hat X \setminus \overline D)$ such that $f(u) = 
\langle v, u\rangle _{{\rm Gr}}$ for all  $u\in S_{A}^{s,p} (D)$;

3) the mapping 
\begin{equation} \label{eq.mapping}
\tilde S_{A^*}^{m-s,p'} (\hat X \setminus \overline D) \ni v \to 
\langle v, \cdot\rangle _{{\rm Gr}}  \in (S_A ^{s,p} (D))^*
\end{equation}  
is a homeomorphism 
of these spaces.

{\bf The pairing.} 
With this purpose, for a fixed $ v \in S_{A^*}^{m-s,p'} (\hat X \setminus \overline D) $, 
let's consider the  functional induced by the pairing \eqref{eq.pairing.Gr.fin}.
It is linear  because the pairings $\langle \cdot, \cdot \rangle_{\partial D}$ are bi-linear.

Using the generalized Cauchy inequality for pairings and \eqref{eq.traces.ext}, we get 
\begin{equation} \label{eq.cont.mapping}
	\left|\langle v, u\rangle _{{\rm Gr}}\right| \leq
	\sum_{j=0}^{m-1}
	\left\| C^A_{m-j-1} v\right\|_{\left[B^{j-s+\frac{1}{p}, p'}(\partial D)\right]^k} 
	\left\| B_{j} u\right\|_{\left[B^{s-j-\frac{1}{p}, p}(\partial D)\right]^k} 
	\leq
\end{equation}
$$
	\leq
	C_1 \left\| v \right\|_{\tilde S^{m-s,p'}_{A^*}(\hat X \setminus \overline D)}
	 \left\| u \right\|_{S^{s,p}_{A}(D)} = C \left\| u \right\|_{S^{s,p}_{A}(D)}
$$
for any $ v \in\tilde S_{A^*}^{m-s,p'} (\hat X \setminus \overline D) $.
Thus, given $ v \in \tilde S_{A^*}^{m-s,p'} (\hat X \setminus \overline D) $, the functional 
$\langle v, \cdot \rangle _{{\rm Gr}}$ is  linear and continuous over $S^{s,p}_{A}(D)$.

{\bf The injectivity.}
Next, suppose that $\langle v, u\rangle _{{\rm Gr}}=0$ for any $u \in S^{s,p}_{A}(D)$. Since $\Phi(x,y) \in 
S^{s,p}_{A}(D)$ with respect to the variable $x \in D$ for any fixed 
$y \in X \setminus \overline D$, we obtain 
$$
-\sum_{j=0}^{m-1} \langle B_j \Phi(\cdot,y)  , {C}^A_{m-1-j} v   \rangle_{\partial D}=0
\mbox{ for all } y \in X \setminus \overline D. 
$$ 
It follows from Lemma \ref{l.Green.A*}, that  $v(y)=0$ for all $ y \in X \setminus \overline D$. 
This means that mapping \eqref{eq.mapping} 
 is injective on $S^{s,p}_{A}(D)$.

{\bf The surjectivity.} 
Now, it follows from Lemma \ref{l.chract.traces.A} and Remark \ref{r.norm.A} that the space 
$S^{s,p}_{A}(D)$ can be considered as a closed subspace of the Banach space 
${\mathcal Y}=\oplus_{j=0}^{m-1}\left[B^{s-j-\frac{1}{p}, p}(\partial D)\right]^k$. 
By the discussion above, 
${\mathcal Y}^* =\oplus_{j=0}^{m-1}\left[B^{-s+j+\frac{1}{p}, p'}(\partial D)\right]^k$.

Then, by Hahn-Banach theorem, for any functional $f \in \left( S^{s,p}_{A}(D) \right)^*$ there is 
an element $\oplus_{j=0}^{m-1} v_j \in \oplus_{j=0}^{m-1} \left[B^{j-s+\frac{1}{p}, p'}(\partial D)\right]^k$ 
such that 
\begin{equation} \label{eq.extension}
f(u) = 	\sum_{j=0}^{m-1} \langle v_j,  
B_j u  \rangle_{\partial D} \mbox{ for all } u \in S^{s,p}_{A}(D).
\end{equation}
Next, using jump theorem for the Green integrals $\tilde {\mathcal G}(\oplus_{j=0}^{m-1} v_j)$ related to 
operator $A^*$ with the Dirichlet system $\{ C^A_{m-j-1}\}$ instead of $\{ B_j\}$ we see that 
\begin{equation} \label{eq.weak.jump.*}
\langle g_{i}, (C^A_{m-i-1}   \tilde {\mathcal G}(\oplus_{j=0}^{m-1} v_j))^-_{|\partial D} -
(C^A_{m-i-1}   \tilde{\mathcal G}(\oplus_{j=0}^{m-1} v_j))^+_{|\partial D} \rangle _{\partial D} = 
\langle g_{i},v_{i} \rangle_{\partial D} 
\end{equation}
for all $\oplus_ {i=1}^{m-1} g_i \in \oplus_ {j=0}^{m-1} [C^{\infty} (\partial D)]^k$. 
Therefore, combining \eqref{eq.extension} and \eqref{eq.weak.jump.*} we obtain 
 
\begin{equation*} 
f(u) = 	\sum_{i=0}^{m-1} \langle (C^A_{m-i-1}   \tilde{\mathcal G}(\oplus_{j=0}^{m-1} v_j))^-_{|\partial D} -
(C^A_{m-i-1}   \tilde {\mathcal G}(\oplus_{j=0}^{m-1} v_j))^+_{|\partial D} ,  
B_i u  \rangle_{\partial D}
\end{equation*}
for all  $u \in S^{s,p}_{A}(D)$.

Finally, as we have seen above,   
$$ \tilde {\mathcal G}(\oplus_{j=0}^{m-1} v_j)^+ \in 
\tilde S_{A^*}^{m-s,p'} (\hat X \setminus \overline D) 
\mbox{ and } \tilde {\mathcal G}(\oplus_{j=0}^{m-1} v_j)^- \in 
 S_{A^*}^{m-s,p'} (D).
$$
 Hence, by the first Green formula,  
$$
\sum_{i=0}^{m-1} \langle (C^A_{m-i-1}   \tilde{\mathcal G}(\oplus_{j=0}^{m-1} v_j))^-_{|\partial D},   
B_i u  \rangle_{\partial D}  = 
$$
$$
\langle \tilde{\mathcal G}(\oplus_{j=0}^{m-1} v_j)^- , A u \rangle_D - 
\langle A^* \tilde{\mathcal G}(\oplus_{j=0}^{m-1} v_j)^- ,  u \rangle_D = 0 
$$
and therefore, 
\begin{equation} \label{eq.f.surj}
f(u)  = \langle v, u\rangle _{{\rm Gr}} \mbox{ for all } u \in S^{s,p}_{A}(D)
\mbox{ with  } v= \tilde{\mathcal G}(\oplus_{j=0}^{m-1} v_j)^+ \in \tilde S_{A^*}^{m-s,p'} 
(\hat X \setminus \overline D), 
\end{equation}
i.e. mapping \eqref{eq.mapping} 
 is surjective on $S^{s,p}_{A}(D)$.

{\bf The homeomorphism.} Moreover, by formula \eqref{eq.cont.mapping},  mapping \eqref{eq.mapping} is 
linear, continuous and its inverse is continuous because of the Open Mapping Theorem. Thus, 
\eqref{eq.mapping} is a homeomorphism of the spaces, that was to be proved.  
\end{proof}
 
Next we discuss  the duals for spaces of solutions of finite order of growth. 
Denote by $( S_{A^*} ^{(F)}  (\hat X \setminus \overline D))^*$ the subset 
of $( S_{A^*}   (\hat X \setminus \overline D))^*$ of solutions of finite order of growth 
near $\partial D$. By the discussion above 
$$
  S_{A^*} ^{(F)} (\hat X \setminus  \overline D) = 
\cup _{s\in {\mathbb Z} \atop 1<p<+\infty} \tilde S_{A^*}^{m-s,p} (\hat X \setminus \overline D)
=   \cup _{s\in {\mathbb Z} } \tilde S_{A^*}^{m-s,2} (\hat X \setminus \overline D) ,
$$
we endow it with the inductive limit topology related to the family of Banach spaces 
$\{ \tilde S_{A^*}^{m-s,2} (\hat X \setminus \overline D) \}_{s\in \mathbb Z}$. 

\begin{cor} \label{c.Groth.FS}
Let $A$ be an elliptic operator on $X$ with the UCP for both $A^*$ and $A$, 
$D$ be a relatively compact domain  in $X$. If $X\setminus D$ has no compact components in $X$ then 

1) the strong dual  $(S_A ^{(F)}  (D))^*$ for the DF-space $S_A^{(F)} (D)$ is topologically isomorphic to 
the  Fr\'echet space $ S_{A^*} (\hat X \setminus  \overline D) \cap C^\infty (X\setminus D)$;

2) the strong dual  $(S_A   (D)  \cap C^\infty (\overline  D))^*$ for the Fr\'echet  
space $S_A (D)  \cap C^\infty (\overline  D)$ is topologically isomorphic to 
the DF-space $  S_{A^*} ^{(F)} (\hat X \setminus  \overline D) $;

3) the strong dual  $(S_{A^*} ^{(F)}  (\hat X \setminus \overline D))^*$ for the DF-space
 $ S_{A^*}^{(F)} (\hat X \setminus \overline D)$ is topologically isomorphic to 
the  Fr\'echet space $ S_{A} (D) \cap C^\infty (\overline D)$.
\end{cor}

\begin{proof} According to the Sobolev Embedding Theorem we have 
$$
 S_{A^*} (\hat X \setminus  \overline D) \cap C^\infty (X\setminus D) 
= \cap _{s\in {\mathbb Z} \atop 1<p<+\infty} \tilde S_{A^*}^{m-s,p} (\hat X \setminus \overline D)
=   \cap _{s\in {\mathbb Z} } \tilde S_{A^*}^{m-s,2} (\hat X \setminus \overline D) 
$$
and the corresponding Fr\'echet topology may be induced by the related Sobolev semi-norms over 
compact subsets in $X\setminus D$. 
Then the statement 1) of the corollary follows immediately from Theorem \ref{t.Groth.sp} and Corollary \ref{c.Green} 
where the related pairing is defined by \eqref{eq.pairing.Gr.fin}. 
The statement 2) can be proved in a similar way. Again, the statement 3)  follows 
because Montel spaces are reflexive and the dual to a Montel space 
is a Montel space, see \cite[\S 5.8]{Shaef}, \cite[p. 74]{RobRob}.
\end{proof}

\begin{prop} \label{p.Montel.2} The spaces $S_{A^*} (\hat X \setminus \overline D)$ and 
$S_{A^*} (\hat X \setminus \overline D)\cap C^\infty ( X\setminus D)$ are Montel Fr\'echet spaces with 
reproducing kernels, the spaces $S_{A^*} (\hat X \setminus  D)$, 
$ S^{(F)}_{A^*} (\hat X \setminus \overline  D)$ are Montel DF-spaces with reproducing kernels, 
$ \tilde S_{A^*}^{s,p} (\hat X \setminus \overline D)$, $s \in \mathbb Z$, $1< p< + \infty$, are Banach spaces  
with reproducing kernel, $\tilde S_{A^*}^{s,2} (\hat X \setminus \overline D) $, $s \in \mathbb Z$, are Hilbert 
spaces with reproducing kernels.
\end{prop}

\begin{cor} \label{c.repr.ker} Let $A$ be an elliptic operator on $X$ with the UCP for both $A^*$ and 
$A$, $D$ be a relatively compact domain  in $X$. If  $X\setminus D$ has no compact components in 
$X$ then the fundamental solution $\Phi$ represents the reproducing kernel for any topological space ${\mathcal B}$ 
from Propositions \ref{p.Montel.1} and \ref{p.Montel.2} in the sense that for each $x\in D$ we have 
\begin{equation} \label{eq.repr.ker}
u(x) =\langle  \Phi^* (x, \cdot), u\rangle _{{\mathcal B}^*, {\mathcal B}} = 
\langle  \Phi^* (x, \cdot), u\rangle _{{\mathrm Gr}}  \mbox{ for all } u \in {\mathcal B}. 
\end{equation}
\end{cor}

\begin{proof} By the definition of the bilateral fundamental solution to $A$, we have 
\begin{equation} \label{eq.fund.A*A}
A_x \Phi (x,y) = 0, \,\, A^*_y \Phi^* (x,y) = 0 \mbox{ for } x\ne y.
\end{equation}
Hence, the columns of the matrix $\Phi (x,y) $  are solutions to the operator $A$  
with respect to the variable $x $ in $D$ and the rows of the matrix $\Phi (x,y) $ are solutions to the 
operator $A^*$  with respect to the variable $y $ in $X\setminus D$, 'regular at infinity'. Thus, 
 the statement follows from Green formula \eqref{eq.Green.2} 
(see Lemma \ref{eq.dual.Dir} and Theorem \ref{t.Green}) and the dualities stated in 
Theorems \ref{t.Groth.Fr} and \ref{t.Groth.sp}, Corollaries \ref{c.Groth.Fr} and \ref{c.Groth.FS}.
\end{proof}

Please, note that for $p=2$ the matrix $\Phi (x,y) $ represents the reproducing kernel 
with respect to the Grothendieck pairing \eqref{eq.pairing.Gr.fin} 
 but not with respect to the inner product of the space $S_A^{s,2} (D)$. A situation where 
\eqref{eq.pairing.Gr.fin} may be identified with an inner product of the space $S_A^{m,2} (D)$ 
is indicated in \cite{NaciShla}. 

\begin{rem} \label{r.smoothness} We finish the section with the note on the smoothness 
of the boundary $\partial D$. Actually, for a fixed number $s$ it could be finite:
$\partial D \in C^s$ if $m=|s| \in \mathbb N$. If $m=1=|s|$ it is sufficient to ask for domains 
with Lipschitz boundaries. But, of course, dealing with the space $S_A ^{(F)}  (D)$ one needs 
domain $D$ to have $C^\infty$-boundary.
\end{rem}

\section{Sparse minimizing in spaces $S^{s,p}_A (D)$.}
\label{s.sparse} 

\subsection{On existence and uniqueness of minimizers}

Let us begin with an  existence theorem for Problem \ref{pr.A.s.p.D}. 

\begin{prop} \label{p.exist.s.p.D} 
Let $A$ be an elliptic operator on $X$ with the UCP for both $A^*$ and $A$, 
$D$ be a relatively compact domain  in $X$ with $\partial D\in C^\infty$ and $s\in {\mathbb Z}$, 
$1<p <+\infty$. Let also the complement $X \setminus D$ have no  compact components in $X$.
Then Problem \ref{pr.A.s.p.D} has at least one  minimizer $ u^{(0)} = u^{(0)}_N \in S_A^{s,p} (D)$. 
Moreover, if the functional $F$ is strictly convex, then the minimizer is unique. 
\end{prop}

\begin{proof} Since   
$F$ is proper,  the infimum of Problem \ref{pr.A.s.p.D} does 
not equal to $+\infty$. Next, as  $F$ is convex, lower semi-continuous and coercive, it is bounded 
from below, see for instance \cite[Theorem 2.11 and Remark 2.13]{BaPr} for reflexive 
Banach spaces, i.e.  infimum of Problem  \ref{pr.A.s.p.D}  is finite. Then the statement 
follows immediately from the weak compactness of balls  in reflexive Banach spaces 
(see Banach-Alaoglu theorem \cite[Theorem 3.15]{Rud_FA}) and weak lower semi-continuity for convex 
lower semicontinuous functionals \cite[Proposition  2.10]{BaPr}. 
\end{proof}

\begin{rem} \label{r.unique}
A normed space is called strictly convex if its balls are strictly convex. 
Note that a norm is not a strictly convex functional even if the space is strictly convex (the strong 
inequality fails on vectors satisfying $u_2 = a\, u_1$ with $a\geq 0$). However, 
if ${\mathcal B}$ is a reflexive strictly convex Banach space then, given  $u_0\in {\mathcal B} $, the 
functional $F (u)= {\mathcal J}\big( \|u - u_0\|_{\mathcal B}\big)$ is strictly convex for any strictly 
increasing strictly convex function ${\mathcal J}: [0,+\infty) \to [0,+\infty)$. In particular, given vector 
$x^{(0)} \in {\mathbb R}^N$, the functional $F (x)= (1/q) \|x - x^{(0)}\|^q_{{\mathbb R}^N_q}$ is 
strictly convex, if $1<q<+\infty$.
\end{rem}

Now, recall that a point $u$ of  a convex set $K$ of a locally convex vector space is called 
an extreme point of $K$ if a convex combination 
$$
u = a u_1 + (1 - a) u _2, \, a\in [0,1], \, u_1 , u_2 \in K,
$$ 
is possible in the situation where $u_0=u_1 = u_2$, only. The set of extreme points of $K$ is usually 
denoted by ${\rm Ext} (K)$. 

In the following statement $B_{s,p,D} (0,1)$ stands for the unit ball in the space $[W^{s,p} (D)]^k$.  

\begin{prop} \label{p.convex.sparse.s.p.D}
Under the hypothesis of Proposition \ref{p.exist.s.p.D}, 
if there exists an element $\hat u \in S^{s,p}_A (D)$ 
 such that  $F$ is continuous at ${\mathcal M}_N \hat u$ then there exists   a 
minimizer $\tilde u^{\sharp} = \tilde u^{\sharp}_N $ of Problem \ref{pr.A.s.p.D} in the following form 
\begin{equation} \label{eq.convex.sparse}
\tilde u^\sharp = \|\tilde u^\sharp\|_{[W^{s,p} (D)]^k} \sum_{i=1}^r \gamma_i u_i ,
\end{equation}
where $1\leq r \leq N$, $0<\gamma_i < 1$ with  
$\sum_{i=1}^r \gamma_i =  1$ and $\{ u_i \}_{i=1}^r \subset 
{\rm Ext} (\overline B_{s,p,D} (0,1))$.
\end{prop}

\begin{proof} Follows from the standard arguments of Convex Analysis, using Krein--Milman theorem 
\cite[Theorem 3.22]{Rud_FA} and  Carath\'eodory theorem \cite{Carath} with the use of Fenchel' 
duality theorem and the related optimality condition, cf.
for instance, \cite[Theorem 3.3]{BreCar} even in a more general situation.
\end{proof} 

In papers \cite{Bart_et_all}, \cite{BreCar}, the statement similar to Proposition \ref{p.convex.sparse.s.p.D}
were used to obtain a sparse decomposition via fundamental solutions for minimization problem in the 
space of distributions of high order of singularity where the set extreme points of the unit ball 
consists of $\delta$-functionals. In our situation,  for $s\geq 0$ the space $S_A^{s,p} (D)$  
can not contain $\delta$-functionals. Moreover, it is strictly convex for $s\geq 0$ 
and then the set of extreme points of the unit ball coincides with the unit sphere. Hence, 
this kind of statements  can not help to prove the existence of 
sparse minimizers to Problem \ref{pr.A.s.p.D} in form \eqref{eq.sparse.u},  in principle.

\subsection{Existence of sparse minimizers}

\label{s.sparse.min}
As we have seen in section \S \ref{s.Grothendieck} (see formulae \eqref{eq.fund.A*A}), 
the fundamental solution  $\Phi (\cdot ,y)$ belongs  to $S_A(U)$ for each $y \not \in U \subset X$.  
This  leads us to 
the following statement, where ${\mathcal L } (\{ b_\beta\}_{\beta\in B})$ stands for 
the  linear hull of a family  $\{ b_\beta\}_{\beta\in B}$ in a linear space.

\begin{lem} \label{l.1}
Under the hypothesis of Proposition \ref{p.exist.s.p.D}, 
we have 
$$
\inf_{u \in S^{s,p}_A (D) } {\mathcal F} (u)  = 
\inf_{u   \in {\mathcal L } \big(\{ \Phi_j (\cdot,y)\}_{1\leq j \leq k, \atop y \not \in \overline D}\big)} 
{\mathcal F} (u) ,
$$
where $\Phi_j (\cdot,y)$ is the $j$-th column of the matrix $\Phi (\cdot,y)$. 
\end{lem}

\begin{proof} Indeed, by the very definitions of the related spaces, we have 
$$
{\mathcal L } (\{ \Phi_j (\cdot,y)\}_{1\leq j \leq k, \atop y \not \in \overline D}) \subset 
S_A (\overline D) \subset S^{s,p}_A (D). 
$$
Under the hypothesis of  this lemma we may use Runge type theorems, 
stating that the space $S_A (\overline D)$ is dense in 
$S^{s,p}_A ( D)$, see \cite[Theorem 8.2.2]{Tark36} for $s< m$ and 
\cite[Theorems 8.1.2 and 8.1.3]{Tark36} for $s\geq m$. 
Moreover,  by \eqref{eq.bilateral} all columns of the matrix 
$\Phi (x,y)$  belong to $S_A (\overline D)$ with respect to $x$ for each fixed $y \not \in \overline D$.
Again, theorems on rational approximation of solutions to elliptic systems imply that 
the linear hull ${\mathcal L } (\{ \Phi_j (\cdot,y)\}_{1\leq j \leq k, \atop y \not \in \overline D})$ is dense 
in $S_A (\overline D)$, see \cite[Ch. 5, Theorem 5.3.2]{Tark36}. 
\end{proof}

\begin{rem}
Actually, it is known that,  under the hypothesis of  this lemma, it is possible to choose a countable set of point 
$\{ y^{(i)} \} \subset X \setminus \overline D$ with no limits point on $\partial D$ such that the linear hull 
${\mathcal L } (\{ \Phi_j (\cdot,y^{(i)})\}_{1\leq j \leq k, \atop i \in {\mathbb N}})$ 
is dense in $S_A (D)$, see \cite[Ch. 5, Theorem 5.3.2]{Tark36}.
This justifies somehow our hypothesis on the existence 
of a sparse minimizer $u^\sharp$ to Problem \ref{pr.A.s.p.D}  
in form \eqref{eq.sparse.u} and results of Section \S \ref{s.N.to.infty}.
\end{rem}

Note that
we have $kN$ uncertain constants  in the summation in 
our conjecture \eqref{eq.sparse.u} related to vectors $\{ \vec{C}_j \}_{j=1}^N \subset {\mathbb R}^{k}$. 
This reflect the fact that for $k>1$ the mapping ${\mathcal M} = {\mathcal M}_N$ may include the data related to 
particular components of a vector $u$ from $S_A ^{s,p}(D)$ but not to the solution $u$ itself. 
For this reason we shift to $kN$ data instead of $N$. 
With this purpose, we fix a set $v^{(\cdot)}_{N}$ of linearly independent elements $\{ v^{(j,r)} \}
_{1 \leq j\leq N, \atop 1\leq r\leq k} \subset (S^{s,p}_A (D))^*$.  
Theorem \ref{t.Groth.sp} yields that this system 
may be uniquely identified within $\tilde S^{m-s,p'}_{A^*} (\hat X \setminus D) $. 

Let us consider the mapping  ${\mathcal M} = {\mathcal M}_{kN} : S^{s,p}_A ( D) 
\to {\mathbb R}^{kN}$ in the following form:
\begin{equation} \label{eq.M.vj}
{\mathcal M} u= \big( \langle u , v^{(1,1)} \rangle, \dots \langle u , 
v^{(1,k)} \rangle, \dots 
\langle u , v^{(N,1)} \rangle , \dots \langle u, v^{(N,k)} \rangle \big) ,\, u \in S^{s,p}_A ( D).
\end{equation}
Now, we note that the columns of the fundamental matrix $\Phi (x,  y^{(j)})$
belong to $S_A (\overline D)\subset S^{s,p}_A ( D)$ if $ y^{(j)}\in X\setminus \overline D$.  
If $ y^{(j)}\in \partial D$ then the columns belong to $S^{s,p}_A ( D)$ if, for instance,  
$|x-y^{(j)}|^{m-n}$ belongs to $[W^{s,p} ( D)]^k$, that we should take into the account 
formulating and proving the following theorem.

\begin{thm} \label{t.sparse.s.p.D.fund.vj} 
Under the hypothesis of Proposition \ref{p.exist.s.p.D}, let 
the mapping 
${\mathcal M}$ be defined by \eqref{eq.M.vj}. If 
there exist an element $\hat u \in S^{s,p}_A (D)$ 
 such that  $F $ is continuous at ${\mathcal M}\hat u$ then,  
given linear  independent vectors $\{ v^{(j,r)}\}_{1\leq j \leq N, \atop 1\leq r \leq k} 
\subset (S^{s,p}_A ( D))^*$, 
there is a sparse minimizer $u^{\sharp} = u^{\sharp} _{kN}$
 to Problem \ref{pr.A.s.p.D} in the form  \eqref{eq.sparse.u}  
with  some set $ y^{(\cdot)} _N$  of points $ y^{(1)},  \dots y^{(N)} $ from $X \setminus 
 D$ and vectors 
$ (\vec {C} ^{(1)}, \dots \vec {C} ^{(N)}) $ from ${\mathbb R}^k$. 
\end{thm}

\begin{proof} First,  recall that 
the Fenchel conjugate for a  proper convex lower-semiconti\-nu\-ous
functional $f$ on ${\mathcal B}$ is defined as 
$$
f^* (v) = \sup_{u\in {\mathcal B}} \big( \langle v,u \rangle - f(u)\big) , \, v \in 
{\mathcal B}^*.
$$
According to the Fenchel duality theorem, 
\begin{equation} \label{eq.Fenchel.dual}
\inf_{u \in S_A ^{s,p} (D)} 
\big({\mathcal N} (u) + F ({\mathcal M} u) \big) = \min_{\vec{h} \in {\mathbb R}^{kN}} \big(
{\mathcal N}^*({\mathcal M}^* \vec{h}  )  + F^* (-\vec{h})  \big),
\end{equation}
where ${\mathcal N} (u)= b_N  
\|u \|_{[W^{s,p} (D)]^k}$, ${\mathcal N}^*$, $F^*$ are Fenchel's dual functionals 
for  ${\mathcal N}$ and $F$, respectively. 
Moreover if $\vec{h}^\sharp \in {\mathbb R}^{kN}$ and $u^{(0)} \in S_A ^{s,p} (D$) are  minimizing elements 
for the right-hand side and the left-hand side of  
\eqref{eq.Fenchel.dual}, respectively, then 
\begin{equation} \label{eq.Fenchel.optimal.0}
{\mathcal M}^* \vec{h}^\sharp \in \partial {\mathcal N} (u^{(0)}), \,  
{\mathcal M}  u^{(0)} \in \partial F^*  (-\vec{h}^\sharp)
\end{equation}
where $\partial {\mathcal N} (u^{(0)}) $ is a subgradient of the functional ${\mathcal N} $
at the point $u^{(0)}$ and $\partial F^* (\vec{h}^\sharp)$
is a subgradient of the functional $F^* $ at the point $\vec{h}^\sharp$, 
see \cite[Theorems 3.53, 3.54]{BaPr} or 
\cite[Theorem 3.3.5]{BL}, \cite[Theorem III.4 .1 and Proposition III.4.1]{EkTe}. 
Besides, the optimality conditions \eqref{eq.Fenchel.optimal.0} imply that 
the following properties are fulfilled:
\begin{equation*} 
\langle {\mathcal M}^* 
\vec{h}^\sharp , u^{(0)}\rangle = {\mathcal N} (u^{(0)}) + {\mathcal N}^* (
{\mathcal M}^*  \vec{h}^\sharp ), 
\end{equation*}
see, for instance, \cite[Proposition 2.33]{BaPr}. 

However, it is well known that  if $f(u) = \| u - u_0\|_{\mathcal B}$ in  a Banach space $\mathcal B$ then the 
Legendre--Fenchel conjugate for the functional $f$ is given by 
\begin{equation} \label{eq.Fenchel.norm}
f^* (v) = 
\left\{ 
\begin{array}{lll} + \infty & {\rm if} & \|v\|_{{\mathcal B}^*} >1, \\
\langle v,u_0\rangle & { \rm if} & \|v\|_{{\mathcal B}^*} \leq 1.
\end{array}
\right. 
\end{equation}
Thus, cf. \cite[Remark 2.13]{BreCar}, for similar arguments in locally convex spaces, 
\begin{equation} \label{eq.Fenchel.optimal.1}
\langle  \vec{h}^\sharp , {\mathcal M} u^{(0)}\rangle = 
\langle {\mathcal M}^* \vec{h}^\sharp , u^{(0)}\rangle = {\mathcal N} (u^{(0)}) 
= b_N\|u^{(0)}\|_{[W^{s,p} (D)]^k}.
\end{equation}

Next, using Theorem \ref{t.Groth.sp} we identify the 
system $\{ v^{(j,r)} \} _{j\in {\mathbb N}, \atop 1\leq r\leq k}$ as a linearly independent set of solutions 
from $\tilde S^{m-s,p'}_{A^*} (\hat X \setminus D) $. 
Then $\langle u , v^{(i,j)} \rangle = \langle u , v^{(i,j)} \rangle _{{\rm Gr}}$ and 
the adjont mapping ${\mathcal M}^*: ({\mathbb R}^{kN})^* 
\to (S^{s,p}_A (D))^*$ is given as follows:
$$
({\mathcal M}^* \vec{h}) (y) = \sum_{j=1}^N  \sum_{r=1}^k v^{(j,r)} (y) h_{j,r} 
\mbox{ for } \vec{h} = (\vec{h}_1, \dots  \vec{h}_N) \in {\mathbb R}^{kN}, 
$$
with $\vec{h}_j = (h_{j,1}, \dots h_{j,k}) \in {\mathbb R}^{k}$, 
because for all $u \in S^{s,p}_A (D)$ we have 
$$
\langle {\mathcal M} u , \vec{h} \rangle_{{\mathbb R}^{kN}} = \sum_{j=1}^N 
\sum_{r=1}^k \langle u , v^{(j,r)}\rangle _{{\rm Gr}} h_{j,r} =
\langle u , \sum_{j=1}^N  V^{(j)}\vec{h}_{j} \rangle_{{\rm Gr}} ,
$$
where $V^{(j)}$ is the $(k\times k)$-matrix with $r-$th column being equal to $v^{(j,r)}$.

We are looking for  a vector $u^\sharp$ in the form  \eqref{eq.sparse.u} 
where $\vec{C} = (\vec {C}  _1, \dots \vec {C} _N) $, $\vec {C} _j \in {\mathbb R}^k$,   
and the points $\{ y^{(j)}\}_{j=1}^N 
\subset X \setminus D$ satisfying the following relations:
\begin{equation} \label{eq.sparse.rel.1vj}
{\mathcal M} u^{\sharp} 
= \sum_{j=1}^{N} {\mathcal M} _{x} \Phi(x,y^{(j)}) \, \vec {C} _j 
= {\mathcal M} u^{(0)}, 
\end{equation}
\begin{equation} \label{eq.sparse.rel.2vj}
\langle  u^{\sharp}, {\mathcal M} ^* \vec{h}^\sharp \rangle_{{\rm Gr}} = 
\langle {\mathcal M}  u^{\sharp}, \vec{h}^\sharp \rangle_{{\mathbb R}^{kN}} = 
b_N \|u^{\sharp} \|_{[W^{s,p} (D)]^k} 
= b_N \|u^{(0)} \|_{[W^{s,p} (D)]^k}
\end{equation}
with a minimizer $u^{(0)}$ of Problem \ref{pr.A.s.p.D} granted by Proposition 
\ref{p.exist.s.p.D} 
and a minimizer $\vec{h}^\sharp$ of Fenchel's dual problem, see \eqref{eq.Fenchel.dual}.

Then, as $\{ y^{(j)}\}_{j=1}^N 
\subset X \setminus D$, by the definition of the fundamental solution, $u^{\sharp}
\in S_A^{s,p} (D)$. 
Next, by the very construction, $F ({\mathcal M}  u ^{(0)}) = F  ({\mathcal M} u^\sharp)$ 
and moreover, by \eqref{eq.Fenchel.optimal.1}, 
$$
b_N \|u^{(0)} \|_{[W^{s,p} (D)]^k} = 
\langle {\mathcal M} u ^{(0)}, h^\sharp \rangle =
\langle {\mathcal M}   u^{\sharp} , h^\sharp \rangle
 = b_N \|u^{\sharp} \|_{[W^{s,p} (D)]^k}
$$
Thus, under \eqref{eq.sparse.rel.1vj} and 
 \eqref{eq.sparse.rel.2vj}, $u^{\sharp} $ is a minimizer of Problem \ref{pr.A.s.p.D}, too, and 
it is left to indicate the sets $\{ y^{(j)} \} \subset X\setminus D$ and 
$\{ \vec{C}^{(j)} \} \subset {\mathbb R}^k$ providing these relations. 

If $u^{(0)} =0$ then we simply set $ u^{\sharp}  = \sum_{j=1}^N  \Phi(x,y^{(j)}) \, 0 = 0$ with any 
$\{ y^{(j)} \} \subset X\setminus  D$. 

Let us  continue with $u^{(0)}\ne 0$.

\begin{lem} \label{l.non.degen.vj}
Let ${\mathcal M} $ be given by \eqref{eq.M.vj} and $X\setminus \overline D$ be connected. 
Given set $v^{(\cdot)}_N$ of linearly independent system of  vectors from 
$ \tilde S^{m-s,p'}_{A^*} ( \hat X\setminus D)$, 
the  determinant $\det W  (y^{(\cdot)}_N, v^{(\cdot)}_N)$, 
$y^{(\cdot)}_N \subset X\setminus \overline D$ of the matrix 
\begin{equation} \label{eq.W.vj}
W (v^{(\cdot)}_N, y^{(\cdot)}_N) = 
\left( 
\begin{array}{ccccc}
 V^{(1)} (y^{(1)}) &  \dots &  V^{(N)} (y^{(1)})  \\
\dots & \dots &  \dots\\
V^{(1)} (y^{(N-1)}) &  \dots &   V^{(N)} (y^{(N-1)})  \\
 V^{(1)} (y^{(N)}) &  \dots &  V^{(N)} (y^{(N)})  \\
\end{array}
\right)
\end{equation}
may vanish, only if 
\begin{equation}\label{eq.degen.set}
( y^{(1)}, \dots y^{(N)} ) \in \Sigma_1 \times \Sigma_2 \times \dots \times \Sigma_N 
=\Sigma \subset ( X \setminus D)^N, 
\end{equation}
where each set $\Sigma_i$ has no open subsets in $X \setminus D$. 
\end{lem}

\begin{proof} 
Applying $(kN \times kN)$-matrix $W(v^{(\cdot)}_N , y^{(\cdot)}_N)$ 
to a vector $\vec{C} = (\vec{C}_1, \dots \vec{C}_N ) \in 
{\mathbb R}^{k N} $ with vector entries $\{ \vec{C}_j \}_{j=1}^N 
\subset {\mathbb R}^{k}$, we obtain a vector 
$$
\vec{\gamma} (y^{(1)}, \dots , y^{(N)}) = (\vec{\gamma}_1 (y^{(1)}), \dots \vec{\gamma}_N (y^{(N)}))
\in {\mathbb R}^{kN}
$$ 
where the $i$-th vector component $\vec{\gamma}_i (y^{(i)})\in {\mathbb R}^k$ 
is a linear combination of columns of matrices 
$V^{(1)} (y^{(i)}), \dots V^{(N)} (y^{(i)})$ (i.e. linear combination of elements from $v_N^{(\cdot)}$):
$$
\vec{\gamma}_i (y^{(i)}) = 
\sum_{j=1}^N V^{(j)} (y^{(i)}) \vec{C}_j .
$$
If this component vanishes for all $y^{(i)}$ from an open set $U_i \subset X \setminus \overline D$ 
then, as $X \setminus \overline D$ is connected and $A^*$ possesses the Unique Continuation 
Property, $\vec{\gamma} (y^{(i)})$ equals to zero on 
$X \setminus \overline D$. Since $v^{(\cdot)}_N$ is linearly independent, we conclude that all 
the vectors $\{ \vec{C}_j\}_{j=1}^N$ equal to zero, i.e. 
the vector $\vec{C}$ equals to zero, too. Hence, 
in this particular case, the component $\vec{\gamma}_i (y^{(i)})$ may vanish 
on a set  $\Sigma_i \subset X \setminus D$ with no open subsets, only, because 
of the  Unique Continuation Property. Therefore, 
the matrix $ W (v^{(\cdot)}_N, y^{(\cdot)}_N)$ may degenerate if 
\eqref{eq.degen.set} is fulfilled, only. 
\end{proof}

So, we may always choose pairwise distinct points $\{ y^{(i)}\}_{i=1}^{N} \subset (X \setminus 
\overline D )^N \setminus \Sigma$, 
where the  determinant $\det W (v^{(\cdot)}_N, y^{(\cdot)}_N)$ does not equal to zero.
But if for some $j$ we have $ y^{(j)}\in \partial D$ then 
we have to check that $\Phi (x,  y^{(j)})$ belongs to $S^{s,p}_A ( D)$, as we have noted above.
Then we set 
\begin{equation} \label{eq.sparse.vj}
u^\sharp (x, y^{(\cdot)}_N) = \sum_{j=1}^N \Phi (x, y^{(j)})
W(v^{(\cdot)}_N, y^{(\cdot)}_N) \big)^{-1} {\mathcal M} u ^{(0)}.
\end{equation}
Clearly, for $i$-th vector component $({\mathcal M})_i \Phi (x, y^{(j)})$ of 
${\mathcal M} \Phi (x, y^{(j)})$ we have 
$$
({\mathcal M})_i \Phi (x, y^{(j)}) = \langle V^{(i)}, \Phi (x, y^{(j)}) \rangle_{{\rm Gr}} = 
V^{(i)} (y^{(j)}) 
$$
because of Theorem \ref{t.Green}. Hence, by the very construction, 
$$
{\mathcal M} u^\sharp  (x, y^{(\cdot)}) = W(v^{(\cdot)}_N , y^{(\cdot)}_N)
\big( W(v^{(\cdot)}_N,  y^{(\cdot)}_N) \big)^{-1} 
{\mathcal M} u ^{(0)} = {\mathcal M} u ^{(0)} ,
$$
and, moreover, 
$$
\langle {\mathcal M} u^\sharp (x, y^{(\cdot)}_N) ,  \vec{h} ^\sharp \rangle_{{\mathbb R}^{kN}} = 
\langle {\mathcal M}  u ^{(0)} , \vec{h} ^\sharp \rangle _{{\mathbb R}^{kN}}= b_N 
\|u^{(0)} \|_{[W^{s,p} (D)]^k}.
$$
Therefore, $u^\sharp  (x,  y^{(\cdot)}_N)$ satisfies relations \eqref{eq.sparse.rel.1vj}, 
\eqref{eq.sparse.rel.2vj} and hence it is a minimizer to Problem \ref{pr.A.s.p.D}.
\end{proof}

It follows from the proof of Theorem \ref{t.sparse.s.p.D.fund.vj}  that sparse representations
of the form \eqref{eq.sparse.u} for minimizers  to Problem \ref{pr.A.s.p.D}
are not uniquely defined even if the minimizer is unique because we may choose different sets 
$\{ y^{(i)}\}_{i=1}^{N}$ with non-degenerate matrices $W(v^{(\cdot)}_N, y^{(\cdot)}_N)$.

\subsection{Passing from $N$ to infinity}
\label{s.N.to.infty} 

Now, let's remind that all the objects in the formulation of  Problem \ref{pr.A.s.p.D} depend on the dimension 
of the data space; in our situation these are $F=F_{kN}$, ${\mathcal F} = {\mathcal F}_{kN}$, 
${\mathcal M} = {\mathcal M}_{kN}$ 
and the minimizers $u ^{(0)} = u ^{(0)}_{kN}$.  

Let $l_q$ be the separable space of  sequences  with the (finite) norm 
$$
\|\vec h \|_{l_q} = \Big( \sum_{\nu=1}^\infty \sum_{r=1}^k  |h_{\nu,r}|^q  \Big)^{1/q}, \, 1\leq q <+\infty,
$$
and let $\mathfrak M $ be the non-separable Banach space of bounded sequences  
endowed with the norm
$$
\|\vec h \|_{\mathfrak M } = \sup_{\nu \in \mathbb N} \max_{1\leq r \leq k} |h_{\nu,r}|  ;
$$
as it is known, 
$l_{ q} \subset l_{\tilde q}$ for $\tilde q\geq  q$ and $l_q \subset {\mathfrak M}$ for all $q\geq 1$. 
 
Note that the space $S^{s,p}_A ( D)$ 
is separable as a subspace of separable Banach space $[W^{s,p} ( D)]^k$ and hence, the Banach space 
$\tilde S^{m-s,p'}_{A^*} (\hat X \setminus D) $ is separable, too, as its dual $S^{s,p}_A ( D)$ is separable, 
\cite[Theorem 4.6-8]{Krey}. Thus, we may consider that  linearly  independent set $v^{(\cdot)}_{N}$ 
is a subset of a greater system $v^{(\cdot)}_{\infty} = \{ v^{(j,r)} \}
_{j\in {\mathbb N}, \atop 1\leq r\leq k} \subset (S^{s,p}_A (D))^*$ such that its linear 
span is dense there.  Without loss of the generality we may assume that 
\begin{equation} \label{eq.bound.vj}
\| v^{(j,r)}\|_{(S^{s,p}_A (D))^*} \leq 1 \mbox{ for all the pairs } (j,r).
\end{equation}

We define the linear mapping ${\mathcal M}_\infty: S_A ^{s,p} (D) \to \mathfrak M $
via 
\begin{equation} \label{eq.M.vj.infty}
{\mathcal M}_{\infty} u= \big( \langle u , v^{(1,1)} \rangle, \dots \langle u , 
v^{(1,k)} \rangle, \dots 
\langle u , v^{(N,1)} \rangle  
\dots \langle u, v^{(N,k)} \rangle \dots  \big). 
\end{equation}
It is continuous because of \eqref{eq.bound.vj}. 

\begin{lem} \label{l.M.infty}
Under the hypothesis of Proposition \ref{p.exist.s.p.D}, let the linear 
span of the linearly independent system $v^{(\cdot)}_{\infty} \subset (S^{s,p}_A (D))^*$ is dense in 
$(S_A ^{s,p} (D))^*$. Then, given 
$\vec{h}^{(0)} \in \mathfrak M$,  the operator equation 
\begin{equation}\label{eq.M.infty}
 {\mathcal M}_{\infty} u = \vec{h}^{(0)}
\end{equation}
can not have more than one solution $u^{(0)}_\infty$ in $S_A ^{s,p} (D)$.
\end{lem} 

\begin{proof} As the mapping is linear, we have to prove that any solution to 
\eqref{eq.M.infty} with  $\vec{h}^{(0)}=0$ is trivial. However this follows from Hahn-Banach Theorem 
because $v^{(\cdot)}_{\infty} $ is everywhere dense in $(S_A ^{s,p} (D))^*$.
\end{proof}

Next, we define $\Pi_{kN} : {\mathfrak M} \to {\mathbb R}^{kN} $ as a projection 
 to ${\mathbb R}^{kN}$ and the trivial extension $E_{kN} : {\mathbb R}^{kN} \to \mathfrak M  $: 
$$
\Pi_{kN} \vec{h} = (\vec h_1, \dots \vec h_N), \,  \vec{h}\in {\mathfrak M} , \, 
E_{kN} \vec{g} = (\vec g_1, \dots \vec g_N, 0 , \dots , 0 , \dots ), \, \vec g \in {\mathbb R}^{kN}.   
$$
Of course, the operators $\Pi_{kN} : l_{q} \to {\mathbb R}^{kN} $ and 
$E_{kN} : {\mathbb R}^{kN} \to l_{q} $, $1\leq q<\infty$, are well defined, too. 
In particular,   
\begin{equation}
\label{eq.proj.M}
\Pi _{kN} {\mathcal M}_{\infty} u = {\mathcal M}_{kN} u \mbox{ for all } u \in S_A ^{s,p} (D).
\end{equation}
Clearly, $\{ E_{kN} {\mathcal M}_{kN} u \} $ converges to $ {\mathcal M}_{\infty} u$ 
in $l_{q}$ if $ {\mathcal M}_{\infty} u \in l_{q}$, $1\leq q<+\infty$,
 but it is not usually the case for the examples considered in the next section. 
However, since $\mathfrak M$ is the strong dual for the space of summable sequences $l_1$ 
we see that $\{ E_{kN} {\mathcal M}_{kN} u \} $ converges to $ {\mathcal M}_{\infty} u$ 
*-weakly in  $\mathfrak M$ for each $u \in S_A ^{s,p} (D)$ because 
$$
\langle {\mathcal M}_{\infty} u -  E_{kN} {\mathcal M}_{kN} u , \vec h \rangle _{{\mathcal M},l_1} =  
\sum_{\nu = N+1 } \sum_{r=1}^k \langle u , v^{(\nu,r)} \rangle h_{\nu,r} , \, 
\vec{h} \in l_1.
 $$

The following statement is somehow typical in the theory of Tikhonov's regularization, except the fact that 
we are dealing with the  weak convergence instead of proving bounds for norms. 

\begin{cor} \label{c.infty}
Under the hypothesis of Proposition \ref{p.exist.s.p.D}  and Lemma \ref{l.M.infty}, let $\vec{h}^{(0)} \in 
\mathfrak M $, $1\leq q \leq +\infty$ and, 
for $N\in \mathbb N$, $\vec g \in {\mathbb R}^{kN}$,  
\begin{equation}
\label{eq.F.norm}
F_{kN} (\vec g) =  
\left\{
\begin{array}{lll} 
(1/q)
\|\vec g - \Pi_{kN} \vec h^{(0)}\|^{\tilde p}_{{\mathbb R}^{kN}_{q}}, & {\rm if } &
1\leq q < +\infty, \\
\|\vec g - \Pi_{kN} \vec h^{(0)}\|_{{\mathbb R}^{kN}_{\infty}}, & {\rm if } &
q = +\infty. \\
\end{array}
\right.
\end{equation}
Let also for the constants $b_N$ from \eqref{eq.F.s.p.D} we have 
 $\lim\limits_{N\to +\infty} b_N =0$. Then \eqref{eq.M.infty}  is solvable if and only if there is a positive constant  $C_0$ such that  
\begin{equation}
\label{eq.F.bound}
{\mathcal F}_{kN}(u^{(0)}_{kN} ) \leq C_0 
\,  b_N  \mbox{ for all } N\in \mathbb N;
\end{equation}
besides, if \eqref{eq.M.infty} admits a solution $u ^{(0)}_\infty$ then 
the sequence 
$\{ u^{(0)}_{kN} \}_{N\in \mathbb N}$ of the 
minimizers to related Problem \ref{pr.A.s.p.D} converges weakly to it in $S^{s,p}_A (D)$. 
\end{cor}

\begin{proof}
Indeed, for this particular choice of the functional $F_{kN}$ the assumptions 
of Theorem \ref{t.sparse.s.p.D.fund.vj} are fulfilled. 

If there is a vector $u^{(0)}_\infty \in S_A ^{s,p} (D)$ satisfying \eqref{eq.M.infty} then, 
by the minimality of $ u^{(0)}_{kN}$ and \eqref{eq.proj.M}, for all $N \in \mathbb N$ we have  
\begin{equation}\label{eq.est.N}
{\mathcal F}_{kN}(u^{(0)}_{kN} ) \leq {\mathcal F}_{kN}(u^{(0)}_\infty ) = 
\end{equation}
$$
b_N \, \|u^{(0)}_\infty \|_{[W^{s,p} (D)]^k} + F_{kN}(u^{(0)}_\infty)  = 
b_N \, \|u^{(0)}_\infty\|_{[W^{s,p} (D)]^k},
$$
i.e. $C_0 = \|u^{(0)}_\infty\|_{[W^{s,p} (D)]^k}$.

Back, let \eqref{eq.F.bound} hold true. Then, in  particular, 
the sequence $\{  u^{(0)}_{kN} \}_{N\in \mathbb N}$ is bounded in $S_A ^{s,p} (D)$. 
As the space $S_A ^{s,p} (D)$ is reflexive, we may extract a weakly convergent to 
a vector $w^{(0)} \in S_A ^{s,p} (D)$ subsequence $\{  u^{(0)}_{kN_l} \}_{N_l\in \mathbb N}$. 
Since the subsequence is bounded  it converges 
weakly to $w^{(0)}$ if and only if 
$$
\langle f , w^{(0)} - u^{(0)}_{kN_l} \rangle \to 0 \mbox{ as } N_l\to +\infty
$$
for any $f$ from a set $L \subset (S^{s,p}_A (D))^*$ such that ${\mathcal L} (L)$ is everywhere dense 
in $(S^{s,p}_A (D))^*$. Taking $L = v^{(\cdot)}_{\infty} $, for all pairs $(j,r)$ we obtain:
\begin{equation}\label{eq.weak.lim1}
\lim_{N_l \to +\infty }\langle w^{(0)} -  u^{(0)}_{kN_l}, v^{(j,r)} \rangle = 
\lim_{N_l \to +\infty } (({\mathcal M}_\infty w^{(0)})_{j,r} - ({\mathcal M}_{kN_l} u^{(0)}_{kN_l} )_{j,r}) =0 .
\end{equation}
On the other hand, \eqref{eq.F.bound} yields for all $ N_l$, $l \in \mathbb N$,
$$
C_0 b_N \geq 
\left\{
\begin{array}{lll}
\max_{1\leq \nu \leq N_l} \max_{1\leq r \leq k}
| h^{(0)}_{j,r} - ({\mathcal M}_{kN_l}u^{(0)}_{kN_l} )_{j,r}|  
& {\rm if } & q=+\infty,\\
(1/q)\sum_{j=1}^N \sum_{r=1}^k  
|\vec h^{(0)}_{j,r} - ({\mathcal M}_{kN_l}u^{(0)}_{kN_l} )_{j,r}|^{q}
& {\rm if } & 1\leq q<+\infty.
\end{array}
\right.
$$
Since $\lim\limits_{N\to +\infty} b_N =0$, passing to the limit with respect to $N_l\to + \infty$, we obtain
\begin{equation}\label{eq.weak.lim2}
\lim_{N_l \to +\infty } ( \vec h^{(0)}_{j,r} - ({\mathcal M}_{kN} u^{(0)}_{kN_l} )_{j,r}) =0 .
\end{equation}
Combining \eqref{eq.weak.lim1} and \eqref{eq.weak.lim2} we conclude that 
${\mathcal M}_\infty w^{(0)} = \vec{h}^{(0)}$, i.e. $w^{(0)} $ is the unique solution to 
\eqref{eq.M.infty} because of Lemma  \ref{l.M.infty}. In particular, the limit of any other 
weakly convergent subsequence of $\{ u^{(0)}_{kN}\} $  coincides with $w^{(0)}$. 
Besides, according to Weak Compactness Principle for reflexive spaces there are no subsequences 
of $\{ u^{(0)}_{kN}\} $ without weakly convergent subsequences and hence the 
sequence $\{ u^{(0)}_{kN}\} $ itself converges weakly to $w^{(0)}$.
\end{proof}

Note that in the case, where $1<q<+\infty$, the minimizer $u_{kN}$ is unique for each $N 
\in \mathbb N$ (see Remark \ref{r.unique}). 

\begin{rem} \label{r.convergence}
Please, don't be misled about the weak convergence of the sequence $\{ u^{(0)}_{kN}\} $. According to Stiltjes-Vitali 
Theorem, any weakly convergent sequence $\{ w_{\nu}\} $ in the space $S_A ^{s,p} (D)$ actually 
converges uniformly on compact subsets in $D$ together with its partial derivatives of any order.
\end{rem}

\section{Some applications}
\label{s.appl}

\subsection{Variational inverse problems related to neural networks}
\label{s.var}
Actually, we refer to 
\cite{Bart_et_all}, 
\cite{CS2001}, \cite{SS}, \cite{Un23A}, and others, for a tutorial 
of application of variational problems to Machine Learning.

One of the possibilities is the following. Let ${\mathcal X}  \times  {\mathbb R}$ 
be a probability space with a distribution $\mu$ and let $u : {\mathcal X} \to {\mathbb R}$
be a function that needs to be 'restored' (estimated) via its samples $\{ u(x^{(j)}) \}_{j=1}^N$ related 
to a training set $\{ x^{(j)} , z^{(j)} \}_{j=1}^N \subset {\mathcal X}  \times  {\mathbb R}$.
More precisely, given a function space $\mathcal T$ and a loss function $L: {\mathbb R}\times  {\mathbb R} \to 
(0,+\infty)$ we are looking for a function $u \in \mathcal T$, minimizing the functional  
\begin{equation*} 
 {\mathcal L} (u) =\int _{{\mathcal X}  \times  {\mathbb R} } L (z, u(x)) \, d\mu (x,z) .
\end{equation*}
However, as the measure $\mu$ is practically known on the training set, only, one passes usually to 
a discreet minimization, i.e. we are looking for a function $f \in \mathcal B$, minimizing 
the functional 
\begin{equation} \label{eq.discrete} 
\hat {\mathcal L} (u) =\sum_{j=1}^N \hat L (z^{(j)}, u(x^{(j)}))  
\end{equation}
with a possibly different function space ${\mathcal B}$ and a discretization $\hat L$ of the 
integrand  $L (z, u(x))  d\mu (x,z)$. Since minimizers of \eqref{eq.discrete} 
might be unstable, a regularizing term ${\mathcal J}(u)$ is usually added to \eqref{eq.discrete} 
for minimization in $\mathcal B$: 
\begin{equation*} \label{eq.discrete.reg} 
 \sum_{j=1}^N \hat L (z^{(j)}, u(x^{(j)}))  + {\mathcal J}(u) .
\end{equation*}
 
In our particular situation, we have ${\mathcal B} = S_A^{s,p}(D)$, 
${\mathcal J}(u)=b_N\|u\|_{[W^{s,p} (D)]^k}$ and the minimized 
functional is given by \eqref{eq.F.s.p.D}. In order to specify the linear mapping  related to data, 
 we define  ${\mathcal M} = {\mathcal M}_{kN}: S^{s,p}_A ( D) \to {\mathbb R}^{kN}$ as follows:
\begin{equation} \label{eq.M.points}
{\mathcal M} u= (  u (x^{(1)}), \dots u (x^{(N)})), \, u \in S^{s,p}_A ( D),
\end{equation}
with a set $x^{(\cdot)}_N$ of some pairwise distinct points $\{ x^{(j)} \}_{j=1}^N \subset D$.

\begin{cor} \label{c.sparse.s.p.D.fund} 
Let ${\mathcal M}$ be given by \eqref{eq.M.points}. 
Under the hypothesis of Proposition \ref{p.exist.s.p.D} and Theorem \ref{t.sparse.s.p.D.fund.vj}, 
given pair-wise distinct points $\{ x^{(i)}\}_{i=1}^N \subset \overline D$, 
there is a sparse minimizer $u^{\sharp} = u^{\sharp}_{kN}$ to Problem \ref{pr.A.s.p.D} in the form  \eqref{eq.sparse.u}  
with  some points $ y^{(1)},  \dots y^{(N)} $ from $X \setminus D$ and vectors 
$ \vec {C} ^{(1)}, \dots \vec {C} ^{(N)} $ from ${\mathbb R}^k$. 
\end{cor}

\begin{proof} Choose the set of functionals $v^{(\cdot)}_N \subset S_A^{s,p} (D)$ consisting 
of $\delta$-functionals related to pair-wise distinct points $\{ x^{(i)}\}_{i=1}^N \subset \overline D$, 
and components of vectors from $S_A^{s,p} (D)$, i.e. $\langle v^{(j,r)} , 
u\rangle = \delta_{x^{(i)}} u_r =  u_r (x^{(i)})$,  
where $r$ is the number of the component. Then, using Green formula \eqref{eq.Green.M.B.2}
we identify $v^{(j,r)}$ with the  
 $r$-th columns $v^{(j,r)} (y) = \Phi_r (x^{(j)},y)$ of the matrix $\Phi (x^{(j)},y)$. 

\begin{lem} \label{l.lin.indep}
Let  ${\mathcal M}$ be given by \eqref{eq.M.points}. 
Given pair-wise distinct points $\{ x^{(i)}\}_{i=1}^N \subset \overline D$, the  
 columns $v^{(j,r)} (y) = \Phi_r (x^{(j)},y)$ of the  matrices  
\begin{equation} \label{eq.matrices}
 V^{(1)} (y)= \Phi (x^{(1)},y) , \dots   , V^{(N)} (y)=  \Phi (x^{(N)},y)  
\end{equation} 
form a  linearly independent 
system of vectors in $S_{A^*}(U)$ for each domain $U \subset X \setminus \{ x^{(j)}\}_{j=1}^N$.
\end{lem}

\begin{proof}
Indeed, set   
$$
v (y) = 
\sum_{i=1}^{N} \Phi(x^{(i)},y) \, \vec {C} _i
$$
with some vectors $\{\vec {C} _1 , \dots \vec {C} _N\} \subset {\mathbb R}^k$. By the properties of 
the bilateral fundamental solution, $v \in S_{A^*}(\hat X \setminus \{ x^{(i)}\}_{i=1}^N)$. According 
to the Unique Continuation Property, if $v \equiv 0$ in $U$, 
then it vanishes on $X \setminus \{ x^{(i)}\}_{i=1}^N$. Now, taking 
a test function $\varphi \in [C^\infty _0 (X)]^k$ supported in a neighbourhood $V(x^{(i_0)})$
of a fixed point $x^{(i_0)}$, satisfying $\{ x^{(i)}\}_{i\ne i_0} \not \in V(x^{(i_0)})$, we see 
that
$$
0 = \int_X v^* (y)  (A  \varphi) (y) dy  =  (\vec C_{i_0})^* \varphi  (x^{(i_0)}) 
$$
because the singularity of the kernel 
$\Phi (x,y)$ on the diagonal $\{x=y\}$ is the same as $|x-y| ^{m-n}$. 
As $i_0$ and $\varphi$ can be chosen arbitrarily, we see that $\vec C_{i}=0$ for all $1\leq i \leq N$.
\end{proof}

Thus, if we set $v^{(\cdot)}_N = \{ v^{(i,r)} (y) = \Phi_r (x^{(j)},y)\}_{1\leq j \leq N, \atop 
1\leq r\leq k}$, then  
it follows from Lemmata \ref{l.non.degen.vj} and \ref{l.lin.indep} that,  
given pair-wise distinct points $\{ x^{(j)}\}_{j=1}^{N} \subset \overline D$ 
 the  determinant $\det W_N (y^{(1)}, \dots , y^{(N)})$ of the matrix 
\begin{equation}\label{eq.W.points}
W_N(v^{(\cdot)}, y^{(\cdot)}) = 
\left( 
\begin{array}{ccccc}
 \Phi (x^{(1)},y^{(1)}) &  \dots &  \Phi (x^{(N)},y^{(1)})  \\
\dots & \dots &  \dots\\
 \Phi (x^{(1)},y^{(N-1)}) &  \dots &   \Phi (x^{(N)},y^{(N-1)})  \\
 \Phi (x^{(1)},y^{(N)}) &  \dots &  \Phi (x^{(N)},y^{(N)})  \\
\end{array}
\right)
\end{equation}
may vanish only if \eqref{eq.degen.set} is fulfilled with the sets $\Sigma_i$, $1\leq i \leq N$,  
that have no open subsets in $X \setminus \overline D$. 

Finally, it follows from Theorem \ref{t.sparse.s.p.D.fund.vj} 
that in this particular situation  
a sparse minimizer $u^\sharp_{kN}$ is given by \eqref{eq.sparse.vj} with 
the matrix $W_N$ as in \eqref{eq.W.points}. 
\end{proof}

In this case equation \eqref{eq.M.infty} is equivalent to the following: given 
the set of pairwise distinct points $\{ x^{(i)}\}_{i=1}^{\infty} \subset \overline D$, and a vector 
$\vec{h}^{(0)} \in {\mathfrak M}$, $\vec{h}^{(0)}= (\vec{h}_1, \dots \vec{h}_N, \dots)$, $\vec{h}_i \in 
{\mathbb R}^k$, find   $u \in S_A ^{s,p} (D)$ such that 
\begin{equation}
\label{eq.M.infty.Machine}
u (x^{(i)}) = \vec{h}_i \mbox{ for all } i\in \mathbb N.
\end{equation}

\begin{cor} \label{c.sparse.s.p.D.infty} 
Let ${\mathcal M} $ be given by \eqref{eq.M.points} and the functional $F$ be given by 
\eqref{eq.F.norm}. 
Under the hypothesis of Proposition \ref{p.exist.s.p.D}, 
let the pairwise distinct points $\{ x^{(i)}\}_{i=1}^{\infty} \subset \overline D$ be dense 
in an open subset $U$ of $D$. Then 1)  equation \eqref{eq.M.infty.Machine} 
can not have more than one solution $u^{(0)}_\infty$ in $S_A ^{s,p} (D)$; 2) if equation \eqref{eq.M.infty.Machine}
has a solution $u^{(0)}_\infty$ in $S_A ^{s,p} (D)$ and  $\lim\limits_{N\to +\infty} b_N =0$ then the sequence of minimizer $\{u^{(0)}_{kN}\}$ 
to Problem \ref{pr.A.s.p.D} converges weakly to it in $S_A ^{s,p} (D)$. 
\end{cor}

\begin{proof} Indeed, if the points $\{ x^{(i)}\}_{i=1}^{\infty} \subset \overline D$ are dense 
in an open subset $U$ of $D$ then any two solutions to \eqref{eq.M.infty} coincide on $U$. 
Then the UCP property for $A$ yields they coincide in $D$.  

According to Corollary \ref{c.infty}, 
in order to prove the second statement  we have to show that 
 the  linear envelope of columns $\{ v^{(j,r)} (y) = \Phi_r (x^{(j)},y)\} $ of the  matrices  
$\{  V^{(j)} (y)= \Phi (x^{(j)},y) \}_{j\in \mathbb N}$ is everywhere dense in $(S^{s,p}_A (D))^*$.
But, as the space $S^{s,p}_A (D)$ is reflexive, this is equivalent to the fact that an 
element $u \in  S^{s,p}_A (D)$, vanishing on $\{ v^{(j,r)} (y) \}$, is identically zero.
On the other hand, if 
$$
 0 = \langle u , V^{(j)} (y)\rangle _{{\rm Gr}} = u(x^{(j)}) \mbox{ for all } j \in \mathbb N,
$$
then $u (x) = 0$ for all $x\in U$ and, by UCP for $A$ we see that $u\equiv 0$ in $D$.
\end{proof}

\begin{rem}
Please, note that under hypothesis of Corollary \ref{c.sparse.s.p.D.infty}, equation 
\eqref{eq.M.infty.Machine} corresponds to the ill-posed problem of 'analytic contunuation': given  
$U\subset D$ and an element $\tilde u \in  S^{s,p}_A (U)$ find $u \in  S^{s,p}_A (D)$ such that 
$u = \tilde u$ on $U$.
\end{rem}

\subsection{The  Dirichlet Problem} 
\label{s.Dir} 
In this section we obtain a sparse minimizer for the well-posed Dirichlet Problem for strongly 
elliptic operators. Namely, one have to consider a (possibly, overdetermined) elliptic 
$(\tilde k \times k)$-matrix operator $\tilde A$ of order $\tilde m$ and 
strongly $(k \times k)$-matrix elliptic operator $A = \tilde A ^* \tilde A$ of order $m= 2\tilde m$.

\begin{prob} \label{pr.Dir}
Given datum $\oplus_{l=0}^{m/2-1} u_j \in \oplus_{l=0}^{m/2-1} [B^{s-l-1/p,p} (\partial D)]^k$
find a vector function $u \in [W^{s,p} (D)]^k$ such that 
\begin{equation*}
\left\{ 
\begin{array}{lcc}
\tilde A^* \tilde Au=0 & in & D, \\
\oplus_{l=0}^{m/2-1} B_j u  = \oplus_{l=0}^{\tilde m-1} u_j  & on & \partial D.
\end{array}
\right.
 \end{equation*}
\end{prob}
Again, if $s\geq m/2$ then $B_j u$ are traces on $\partial D$ and 
for $s<m$ we may treat $B_j u$ as weak boundary values on $\partial D$. 
As it is known, this problem is well-posed under the assumptions above, \cite[\S 2.4]{ShTaBDRK} (the UCP is very 
important here!).

In order to define a 'regularization' of Problem \ref{pr.Dir} with the use of Problem \ref{p.exist.s.p.D} and 
Theorem \ref{t.sparse.s.p.D.fund.vj}, 
we  pick a system of linearly independent elements 
\begin{equation} \label{eq.vj.weak.Dir}
\{ \oplus_{l=0}^{m/2-1} w^{(i)}_l \} _{i=1}^{N}  
 \subset \oplus_{l=0}^{m/2-1}[B^{l-s+1/p,p'} (\partial D)]^k. 
\end{equation}
Then we define the mapping ${\mathcal M} = {\mathcal M}_{mkN/2}: S^{s,p}_A ( D) \to {\mathbb R}^{mkN/2}$ similarly to 
\eqref{eq.M.vj}:  
\begin{equation} \label{eq.M.vj.weak.Dir}
{\mathcal M} u= ( \langle\oplus_{l=0}^{m/2-1}\langle B_l u , w^{(1)}_l \rangle_{\partial D}, \dots 
\oplus_{l=0}^{m/2-1}\langle B_l u , w^{(N)}_l \rangle _{\partial D}), \, 
u \in S^{s,p}_A ( D).
\end{equation}
 
\begin{cor} \label{c.sparse.s.p.D.fund.vj.weak.Dir} 
Let $w^{(\cdot)}$ be a linearly independent system \eqref{eq.vj.weak.Dir} and let 
${\mathcal M}$ be given by \eqref{eq.M.vj.weak.Dir}.
Then, under the hypothesis of Proposition \ref{p.exist.s.p.D}, 
there is a sparse minimizer $u^{\sharp}_{kN}$ to Problem 
\ref{pr.A.s.p.D} in the form  \eqref{eq.sparse.u}  
with  some points $\{ y^{(j,l)} \}_{1\leq j \leq N \atop 
0 \leq l \leq m/2-1} $ from $X \setminus D$ and $\{ \vec{C}^{(j,l)} \}_{1\leq j \leq N \atop 
0 \leq l \leq m/2-1} $  from ${\mathbb R}^k$. 
\end{cor}

\begin{proof} Set $w^{(i)}_l=0 $ for $m/2\leq l \leq m-1$. 
As we have seen in the proof of the surjectivity in  Theorem \ref{t.Groth.sp}  (see the arguments from \eqref{eq.extension} to 
the end of the proof), 
$$
\oplus_{l=0}^{m-1}\langle B_l u , w^{(i)}_l \rangle_{\partial D} = 
 \oplus_{l=0}^{m-1} \langle B_l u , 
\tilde {\mathcal G} (\oplus _{\nu=0}^{m-1} w^{(i)}_\nu)^+ 
\rangle_{\partial D}
$$
and hence $\oplus_{l=0}^{m-1}\langle B_l u , w^{(i)}_l \rangle_{\partial D}$ can be identified 
within $\tilde S_{A^*} ^{m-s,p'}(\hat X\setminus \overline D)$ as  
$$
\oplus_{l=0}^{m-1}\langle B_l u , v^{(i)}_l \rangle_{{\rm Gr}} \mbox{ with }
v^{(i)} = \tilde {\mathcal G} (\oplus_{\nu=0}^{m-1} w^{(i)}_\nu)^+,
$$
where of course,  $v^{(i)} \in  \tilde S^{m-s,p'}_{A^*} 
(\hat X\setminus  D) $, $1\leq i \leq N$. Let us see that the system 
$\{v^{(i)}\}$ is linearly independent $\tilde S^{m-s,p'}_{A^*} 
(\hat X\setminus \overline D)$. Again, if we set   
\begin{equation}\label{eq.Green.Dir}
v (y) = 
\sum_{i=1}^{N} v^{(i)} \, \vec {C} _i 
=  \tilde {\mathcal G} \Big(\oplus_{\nu=0}^{m-1} \Big(\sum_{i=1}^{N} w^{(i)}_\nu \vec {C} _i \Big) \Big)^+
\end{equation}
with some vectors $\{\vec {C} _1 , \dots \vec {C} _N\} \subset {\mathbb R}^k$. As $X\setminus \overline D$ is 
connected then to the Unique Continuation Property, if $v \equiv 0$ in $U$ in $X\setminus \overline D$ then it 
vanishes on $X \setminus \overline D$. Now, applying Lemma \ref{l.chract.traces.A} to 
the Green operator $\tilde {\mathcal G}$ for the operator $A^*$ instead of the Green operator ${\mathcal G}$ the 
operator $A$, we conclude that there is a vector $w \in S_{A^*}^{m-s,p'} (D)$ such that 
\begin{equation}\label{eq.Cau.Dir}
\oplus _{\nu=0}^{m/2-1} B_\nu w =0 \mbox{ on } \partial D,  \, \oplus _{\nu=m/2}^{m-1} B_\nu w =
\sum_{i=1}^{N} w^{(i)}_\nu \vec {C} _i  
\mbox{ on } \partial D.
\end{equation}
In particular,  $w$ is a solution to the Dirichlet Problem 
\begin{equation*}
\left\{ 
\begin{array}{lcc}
\tilde A^* \tilde A w=0 & in & D, \\
\oplus_{l=0}^{m/2-1} B_j w  = 0 & on & \partial D.
\end{array}
\right.
 \end{equation*}
and hence, under the Unique Continuation Property, $w=0$ in $D$, see \cite[Theorem 2.26]{ShTaBDRK}. 
Finally, according to \eqref{eq.Cau.Dir}, 
$$
\sum_{i=1}^{N} w^{(i)}_\nu \vec {C} _i   = 0 \mbox{ on } \partial D.
$$
As system \eqref{eq.vj.weak.Dir} is linearly independent, we conclude that $\vec {C} _i = 0$ 
for all $1\leq i \leq N$, i.e. the system $\{v^{(i)} \}$ is linearly independent in 
$\tilde S^{m-s,p'}_{A^*} (\hat X\setminus \overline D)$.

Thus, the statement follows from Theorem \ref{t.sparse.s.p.D.fund.vj}.
\end{proof}

Note that if $s>m/2-1+n/p$ then, by the Sobolev Embedding Theorem, $S_A^{s,p} (D) \subset 
[C^{m/2-1} (\overline D)]^k$ and  then $B_j u \in [C^{m/2-j-1} (\partial D)]^k$, i.e. we may calculate 
them point-wisely on $\partial D$. Then we define 
\begin{equation} \label{eq.M.points.Bj.Dir}
{\mathcal M} u= \big( \oplus_{l=0}^{m/2-1} ( B_l u) (x^{(1,l)}) , \dots 
\oplus_{l=0}^{m/2-1}( B_l u) (x^{(N,l)}) \big) ,
\end{equation}
with a set $x^{(\cdot)}_N$ of some pair-wise distinct points $\{ x^{(l,j)} \}_{1\leq j \leq N \atop 
0 \leq l \leq m/2-1} \subset \partial D$.

\begin{cor} \label{c.sparse.s.p.D.fund.Bj.Dir} 
Let $s>m/2-1+n/p$ and ${\mathcal M} = {\mathcal M}_{kNm/2}$ be given by \eqref{eq.M.points.Bj.Dir}.
Under the hypothesis of Proposition \ref{p.exist.s.p.D}, 
given pair-wise distinct points $\{ x^{(i,l)} \}_{1\leq i \leq N \atop 
0 \leq l \leq m/2-1}\subset \partial D$,
there is a sparse minimizer 
$u^{\sharp}= u^{\sharp}_{kNm/2}$ to Problem \ref{pr.A.s.p.D} 
in the form  \eqref{eq.sparse.u}  
with  some points $\{ y^{(j,l)} \}_{1\leq j \leq N \atop 
0 \leq l \leq m/2-1} $ from $X \setminus D$ and $\{ \vec{C}^{(j,l)} \}_{1\leq j \leq N \atop 
0 \leq l \leq m/2-1} $  from ${\mathbb R}^k$. 
\end{cor}

\begin{proof} Let us show that the composition of $\delta$-functionals, related to 
pair-wise distinct points $\{ x^{(j,l)} \}_{1\leq j \leq N \atop 
0 \leq l \leq m/2-1}\subset \partial D$, with action of differential operators
$B_j$, $0 \leq j \leq {m/2-1}$,  are linearly independent in 
$(S_A^{s,p} (D))^*$. 
Indeed, let 
$$
\sum_{j=1}^N  \sum_{l=0}^{m-1} c_{j,l} ( B_l u) (x^{(j,l)}) =0 \mbox{ for all } u \in S_A^{s,p} (D)
$$
with some real numbers $c_{j,l}$. 
Taking a test function $\varphi_{i_0,l_0} \in [C^\infty  (\partial D)]^k$ supported in a neighbourhood 
$V(x^{(i_0, l_0)})$
of a fixed point $x^{(i_0,l_0)}$, satisfying $\{ x^{(i,l)}\}_{(i,l)\ne (i_0,l_0)} \not \in V(x^{(i_0,l_0)})$, and using the fact that Problem \ref{pr.Dir} is well-posed we see 
that there is an element $u_0 \in S_A^{s,p} (D)$ such that $B_l u_0 = 0$ on $\partial D$ if $l\ne l_0$
and $B_{l_0} u_0 = \varphi_{i_0,l_0} $ on $\partial D$. As we may alway choose 
$\varphi_{i_0,l_0} (x^{(i_0, l_0)}) \ne 0$ we see that 
$$
0=\sum_{j=1}^N  \sum_{l=0}^{m/2-1} c_{j,l} ( B_l u_0) (x^{(j,l)}) = c_{j_0,l_0} \varphi_{i_0,l_0} 
(x^{(i_0, l_0)}) 
$$
and hence $c_{j_0,l_0} = 0$ with arbitrary $(j_0,l_0)$. Thus the statement of this corollary follows from  Corollary \ref{c.sparse.s.p.D.fund.vj.weak.Dir}.
\end{proof}

In this case operator equation \eqref{eq.M.infty} is equivalent to the following: given 
the set of pairwise distinct points $\{ x^{(i)}\}_{i=1}^{\infty} \subset \overline D$, and a vector 
$\vec{h}^{(0)} \in {\mathfrak M}$, 
$$\vec{h}^{(0)}= ( \oplus_{l=0}^{m/2-1}  \vec{h}_{1,l}, \dots 
\oplus_{l=0}^{m/2-1}  \vec{h}_{N,l}, \dots), \, \vec{h}_{i,l} \in 
{\mathbb R}^k,
$$ 
 find   $u \in S_A ^{s,p} (D)$ such that 
\begin{equation}
\label{eq.M.infty.Dir}
\oplus_{l=0}^{m/2-1} B_l u  (x^{(i,l)}) = \oplus_{l=0}^{m/2-1} \vec{h}_{i,l} \mbox{ for all } i\in \mathbb N.
\end{equation}

\begin{cor} \label{c.sparse.Dir.infty} 
Let $s>m/2-1+n/p$, ${\mathcal M}$ be given by \eqref{eq.M.points.Bj.Dir}, the functional $F$ be given by 
\eqref{eq.F.norm} and  $\lim\limits_{N\to +\infty} b_N =0$. 
Under the hypothesis of Proposition \ref{p.exist.s.p.D},  
let the pair-wise distinct points $\{ x^{(i,l)}\}_{i\in \mathbb N \atop 0\leq l \leq m/2-1} \subset 
\partial D$ be chosen in such a way that for each $0\leq l\leq m/2-1$ the set 
$\{ x^{(i,l)}\}_{i\in \mathbb N \atop 0\leq l \leq m/2-1}$  be dense in $\partial D$. 
Then 1)  equation \eqref{eq.M.infty.Dir} 
can not have more than one solution $u^{(0)}_\infty$ in $S_A ^{s,p} (D)$; 2) if 
equation \eqref{eq.M.infty.Dir} has a solution $u^{(0)}_\infty$ in $S_A ^{s,p} (D)$ then the sequence of 
minimizers $\{u^{(0)}_{kNm/2}\}$ to Problem \ref{pr.A.s.p.D} converges weakly in $S_A ^{s,p} (D)$ to it; 
3) if $ \oplus_{l=0}^{m/2-1} \vec{h}_{i,l} = \oplus_{l=0}^{m/2-1}  u_l (x^{(i,l)})$ for some set 
$\oplus_{l=0}^{m/2-1} u_l  \in \oplus_{l=0}^{m/2-1} [B^{s-l-1/p,p} (\partial D)]^k $, $i \in \mathbb N$, then the  sequence $\{u^{(0)}_{kNm/2}\}$ converges to the unique solution  $u^{(0)}$ to Dirichlet Problem \ref{pr.Dir}. 
\end{cor}

\begin{proof} 
According to  Lemma \ref{l.M.infty} and Corollary \ref{c.infty}, 
in order to prove the first and the second statements,  we have to show that 
 the  linear envelope of functionals $\{ f_{i,l}  \}_{i\in \mathbb N \atop 0\leq l \leq m/2-1}$ induced 
by the operators $\{ ( B_l u) (x^{(i,l)}) \}_{i\in \mathbb N \atop 0\leq l \leq m/2-1}$ on 
the space $S_A ^{s,p} (D)$
is everywhere dense in $(S^{s,p}_A (D))^*$.
But, as the space $S^{s,p}_A (D)$ is reflexive, this is equivalent to the fact that an 
element $u \in  S^{s,p}_A (D)$, vanishing on $\{ f_{i,l}  \}_{i\in \mathbb N \atop 0\leq l \leq m/2-1}$
is identically zero.
On the other hand, if points $\{ x^{(i,l)}\}_{i=1}^{\infty} \subset \partial D$ are dense 
in a $\partial D$ for each $l$, $0\leq l \leq m/2-1$,
 then for any solutions $u$, satisfying $\oplus_{l=0}^{m/2-1} B_j u (x^{(i,l)})=0$ 
for all $i \in \mathbb N$ we see that $\oplus_{l=0}^{m/2-1} B_j u =0$ on $\partial D$. Hence, by the Uniqueness Theorem for the Dirichlet Problem \ref{pr.Dir}
see, for instance, \cite[\S 2.4]{ShTaBDRK}, $u \equiv  0$  in $D$.

Finally, as we have mentioned before Problem \ref{pr.Dir} has a unique solution $u^{(0)} \in S_A ^{s,p} (D)$ 
for each data $\oplus_{l=0}^{m/2-1} u_l \in \oplus_{l=0}^{m/2-1} [B^{s-l-1/p,p} (\partial D)]^k$. 
This means that operator equation \eqref{eq.M.infty.Dir} is solvable and $u^{(0)}$ is its unique solution.
\end{proof} 

Of course, we may do the same in the weak sense. In this case operator equation \eqref{eq.M.infty} is equivalent to the following: given 
the set 
$$
w^{(\cdot)}_\infty  = \{ \oplus_{l=0}^{m/2-1} w^{(i)}_l \} _{i \in \mathbb N}  
 \subset \oplus_{l=0}^{m/2-1}[B^{l-s+1/p,p'} (\partial D)]^k
$$ 
 and a vector 
$\vec{h}^{(0)} \in {\mathfrak M}$, 
$$\vec{h}^{(0)}= ( \oplus_{l=0}^{m/2-1}  \vec{h}_{1,l}, \dots 
\oplus_{l=0}^{m/2-1}  \vec{h}_{N,l}, \dots), \, \vec{h}_{i,l} \in 
{\mathbb R}^k,
$$ 
 find   $u \in S_A ^{s,p} (D)$ such that 
\begin{equation}
\label{eq.M.infty.Dir.weak}
\oplus_{l=0}^{m/2-1}\langle B_l u , w^{(i)}_l \rangle_{\partial D}= \oplus_{l=0}^{m/2-1} \vec{h}_{i,l} \mbox{ for all } i\in \mathbb N.
\end{equation}

\begin{cor} \label{c.sparse.Dir.infty.weak} 
Let ${\mathcal M}$ be given by \eqref{eq.M.vj.weak.Dir} and the functional $F$ be given by 
\eqref{eq.F.norm} and  $\lim\limits_{N\to +\infty} b_N =0$. 
Under the hypothesis of Proposition \ref{p.exist.s.p.D},  
let $w^{(\cdot)}_\infty$ be a linearly independent system with its linear envelope 
being dense in  $\oplus_{l=0}^{m/2-1} [B^{s-l-1/p,p} (\partial D)]^k $. 
Then 1)  equation \eqref{eq.M.infty.Dir.weak} 
can not have more than one solution $u^{(0)}_\infty$ in $S_A ^{s,p} (D)$; 2) if 
equation \eqref{eq.M.infty.Dir.weak} has a solution $u^{(0)}_\infty$ in $S_A ^{s,p} (D)$ then the sequence of 
minimizers $\{u^{(0)}_{kNm/2}\}$ to Problem \ref{pr.A.s.p.D} converges weakly in $S_A ^{s,p} (D)$ to it; 
3) if $ \oplus_{l=0}^{m/2-1} \vec{h}_{i,l} = \oplus_{l=0}^{m/2-1}\langle u_l  , w^{(1)}_l \rangle_{\partial D}$
 for some set 
$\oplus_{l=0}^{m/2-1} u_l  \in \oplus_{l=0}^{m/2-1} [B^{s-l-1/p,p} (\partial D)]^k $, $i \in \mathbb N$, then the  sequence  $\{u^{(0)}_{kNm/2}\}$ converges to the unique solution  $u^{(0)}$ to Dirichlet Problem \ref{pr.Dir}. 
\end{cor}

\subsection{The Cauchy problem}
\label{s.Cau}

In a similar way,
 we consider the ill-posed Cauchy Problem for elliptic operators.
First, we pick a non-empty relatively open connected set $\Gamma\subset \partial D$ such that 
$\partial \Gamma$ is smooth and $\partial D \setminus \overline \Gamma$ is non-empty, relatively 
open, too. We also fix a Dirichlet system $\{ B_j\}_{j=0}^{m-1}$ on $\partial D$. 

If $s\geq m$ then, the space $\oplus_{l=0}^{m-1} u_l\in \oplus _{l=0}^{m-1} B^{s-l-1/p,p} (\Gamma)$ can be 
defined with the use of the induced Lebesgue measure on $\partial D$. By Whitney type Theorem, any element 
$\oplus_{l=0}^{m-1} u_l\in \oplus _{l=0}^{m-1} B^{s-l-1/p,p} (\Gamma)$ can be extended up to an element 
$\oplus_{l=0}^{m-1} \tilde u_l\in B^{s-l-1/p,p} (\partial D)$ preserving the norm of the original element. Two 
such extensions $\oplus_{l=0}^{m-1}\tilde u_l$ and $\oplus_{l=0}^{m-1} \hat u_l$ are equivalent in the sense that 
their difference equals to zero on $\Gamma$. Thus the space $\oplus_{l=0}^{m-1} B^{s-l-1/p,p} (\Gamma)$ can be 
naturally treated as a factor space $\oplus_{l=0}^{m-1}\tilde B^{s-l-1/p,p} (\Gamma) = 
\oplus_{l=0}^{m-1}B^{s-l-1/p,p} (\partial D)/B^{s-l-1/p,p}_\Gamma (\partial D)$ with the standard factor norm, 
where $\oplus_{l=0}^{m-1} B^{s-l-1/p,p}_\Gamma (\partial D)$ is the closed subspace of 
$\oplus_{l=0}^{m-1} B^{s-l-1/p,p} (\partial D)$ 
consisting of functions that equal to zero on $\overline \Gamma$.

For $s<m$ we take it for the definition, i.e. 
$$
\oplus_{l=0}^{m-1} \tilde B^{s-l-1/p,p} (\Gamma) = 
\oplus_{l=0}^{m-1} B^{s-l-1/p,p} (\partial D)/B^{s-l-1/p,p}_\Gamma (\partial D), \, s \in \mathbb Z.
$$

\begin{prob} \label{pr.Cau}
Given datum $\oplus_{l=0}^{m-1} u_l \in \oplus_{l=0}^{m-1} [\tilde B^{s-l-1/p,p} (\Gamma)]^k$
find $u \in [W^{s,p} (D)]^k$ such that 
\begin{equation*}
\left\{ 
\begin{array}{lcc}
Au=0 & in & D, \\
\oplus_{l=0}^{m-1} B_l u  = \oplus_{l=0}^{m-1} u_l  & on & \Gamma.
\end{array}
\right.
 \end{equation*}
\end{prob}
Of course, if $s\geq m$ then $B_j u$ are traces on $\partial D$ and 
for $s<m$ we may treat $B_j u$ as weak boundary values on $\partial D$.

As it is known, this problem is ill-posed. However, there are many publications related 
to the problem for different operators in different spaces; moreover, it plays an essential role in many 
applications (Geophysics, Electrodynamics, 
Cardiology, etc.), see, for instance \cite{LvRShi}, \cite{Tark36},  
and their bibliography.  Under assumptions above, it has no more than one solution, 
see \cite[Theorem 2.8]{ShTaLMS}, \cite[Ch. 10]{Tark36} and  
it is densely solvable, see, for instance, \cite[Lemma 3.3]{ShTaLMS} for 
the problem in Hardy classes in $D$  or we can easily prove with the use of Duality Theorem \ref{t.Groth.sp}. 

\begin{prop} \label{p.densely}
Under the hypothesis of Proposition \ref{p.exist.s.p.D}, let  $\Gamma\subset \partial D$ a non-empty 
relatively open connected set such that 
$\partial \Gamma$ is smooth and $\partial D \setminus \overline \Gamma$ is non-empty, 
relatively open, too. If $1<p<+\infty$ then Problem \ref{pr.Cau} is densely solvable.
\end{prop}

\begin{proof} It is known that the dual $({\mathcal B}/ {\mathcal B}_0)^*$  for the factor space 
${\mathcal B}/ {\mathcal B}_0$ for a closed subspace ${\mathcal B}_0$ in ${\mathcal B}$ can be treated 
as subspace of ${\mathcal B}^*$ consisting of functionals vanishing on ${\mathcal B}_0$, see 
\cite[Ch.4]{Shaef}. Hence $(\oplus_{l=0}^{m-1} \tilde B^{s-l-1/p,p} (\Gamma))^*$ may be identified  as 
the subspace in  $\oplus_{l=0}^{m-1} B^{l+1/p-s,p} (\partial D)$ with elements $v_l$ satisfying 
\begin{equation} \label{eq.factor}
\sum _{l=0}^{m-1}\langle v_l, u_l \rangle_{\partial D} = 0 \mbox{ for all } 
\oplus_{l=0}^{m-1} u_l \in 
\oplus_{l=0}^{m-1} B^{s-l-1/p,p}_\Gamma (\partial D). 
\end{equation}
By the Hahn--Banach Theorem, the mapping 
$$
\oplus _{l=0}^{m-1} B_l : S_A^{s,p} (D) \to 
\oplus_{l=0}^{m-1} \tilde B^{s-l-1/p,p} (\Gamma)
$$ 
has dense range if and only if 
any functional from $\big(\oplus_{l=0}^{m-1} \tilde B^{s-l-1/p,p} (\Gamma)\big)^*$,  vanishing 
on $\oplus _{l=0}^{m-1} B_lu$ for any   $u \in S_A^{s,p} (D)$, is identically zero. So, 
let $\oplus_{l=0}^{m-1} v_l \in \oplus_{l=0}^{m-1} B^{l+1/p-s,p} (\partial D)$ satisfies 
\eqref{eq.factor} and 
\begin{equation*} 
\sum _{l=0}^{m-1}\langle v_l, B_lu  \rangle_{\partial D} = 0 \mbox{ for all } 
u \in S_A^{s,p} (D).
\end{equation*}
Then, for any function $\varphi \in C^{\infty} (\partial D)$ being equal to $1$ on $\overline \Gamma$ 
and vanishing on an open set $U\subset \partial D\setminus \overline \Gamma$, we have 
$(1-\varphi) \oplus_{l=0}^{m-1} B_l u \in 
\oplus_{l=0}^{m-1} B^{s-l-1/p,p}_\Gamma (\partial D)$, i.e. 
relation 
\eqref{eq.factor} yields
\begin{equation} \label{eq.dense.1}
\sum _{l=0}^{m-1}\langle \varphi v_l, B_lu  \rangle_{\partial D} = 0 \mbox{ for all } 
u \in S_A^{s,p} (D).
\end{equation} 
Then, as we have seen proving the surjectivity in Theorem \ref{t.Groth.sp} 
(formulae \eqref{eq.extension}, \eqref{eq.f.surj}) 
\begin{equation} \label{eq.dense.2}
 \sum _{l=0}^{m-1}\langle \varphi v_l, B_lu  \rangle_{\partial D} = 
\langle (\tilde{\mathcal G}(\oplus_{j=0}^{m-1} \varphi v_j))^+, u\rangle _{{\mathrm Gr}} 
\mbox{ for all } 
u \in S_A^{s,p} (D).
\end{equation}  
Hence, combining \eqref{eq.dense.1}, \eqref{eq.dense.2}
and  the injectivity part of Theorem \ref{t.Groth.sp} we conclude that 
$(\tilde{\mathcal G}(\oplus_{j=0}^{m-1} \varphi v_j))^+=0 $ in $X\setminus D$. But, as  the function 
$\varphi$ vanishes on an open set $U\subset \partial D\setminus \overline \Gamma$ 
and $X\setminus D$ is connected, 
we see that $\tilde{\mathcal G}(\oplus_{j=0}^{m-1} \varphi v_j)) = 0$ on $X \setminus 
\overline \Gamma$. In particular, jump formula \eqref{eq.weak.jump.*} for the integral 
$\tilde{\mathcal G}(\oplus_{j=0}^{m-1} \varphi v_j))$ yields $\oplus_{j=0}^{m-1} \varphi v_j = 0$ on 
$\partial D$, and then $\oplus_{j=0}^{m-1}  v_j = 0$ on $\partial D$ by \eqref{eq.factor},  
that was to be proved.
\end{proof}

Of course, the regularization methods are very useful for this problem, \cite{Tikh77}, \cite{LvRShi}.
In order to define a 'regularization' of Problem \ref{pr.Cau} 
with the use of Problem \ref{pr.A.s.p.D} and 
Theorem \ref{t.sparse.s.p.D.fund.vj}, 
we have to define a suitable mapping, 
related to the data. 

With this purpose, pick 
a system of linearly independent elements 
\begin{equation} \label{eq.vj.weak}
w^{(\cdot)}_N = \{ \oplus_{l=0}^{m-1} w^{(i)}_l \} _{i=1}^{N}  
 \subset \big(\oplus_{l=0}^{m-1}[\tilde B^{s-l-1/p,p} (\Gamma)]^k)\big)^*.
\end{equation}

Then we define the mapping ${\mathcal M}= {\mathcal M}_{kNm} : S^{s,p}_A ( D) \to {\mathbb R}^{mkN}$  similarly to \eqref{eq.M.vj}: 
\begin{equation} \label{eq.M.vj.weak}
{\mathcal M} u= ( \oplus_{l=0}^{m-1}\langle B_l u , w^{(1)}_l \rangle_{\partial D}, \dots 
\oplus_{l=0}^{m-1}\langle B_l u , w^{(N)}_l \rangle _{\partial D}),\, u \in S^{s,p}_A ( D). 
\end{equation}

\begin{cor} \label{c.sparse.s.p.D.fund.vj.weak} 
Let $w^{(\cdot)}_N$ be a linearly independent system \eqref{eq.vj.weak} in the space 
$\big(\oplus_{l=0}^{m-1}[\tilde B^{s-l-1/p,p} (\Gamma)]^k \big)^* $ 
and let ${\mathcal M}$ be given by \eqref{eq.M.vj.weak}. 
Then, under the hypothesis of Proposition \ref{p.exist.s.p.D}, 
there is a sparse minimizer $u^{\sharp}_{kNm}$ to Problem \ref{pr.A.s.p.D} in the form  \eqref{eq.sparse.u}  
with  some points $\{ y^{(l,j)} \}_{1\leq j \leq N \atop 
0 \leq l \leq m-1} $ from $X \setminus D$ and $\{ \vec{C}^{(l,j)} \}_{1\leq j \leq N \atop 
0 \leq l \leq m-1} $  from ${\mathbb R}^k$. 
\end{cor}

\begin{proof} As we have seen in the proof of Proposition \ref{p.densely}, see formula 
\eqref{eq.dense.2},
$$
\oplus_{l=0}^{m-1}\langle B_l u , w^{(i)}_l \rangle_{\partial D} = 
\oplus_{l=0}^{m-1}\langle B_l u ,\varphi w^{(i)}_l \rangle_{\partial D} = 
 \oplus_{l=0}^{m-1} \langle B_l u , 
\tilde {\mathcal G} (\oplus _{\nu=0}^{m-1} \varphi w^{(i)}_\nu)^+ 
\rangle_{\partial D}
$$
for any function $\varphi \in C^{\infty} (\partial D)$ being equal to $1$ on $\overline \Gamma$ 
and vanishing on an open set $U\subset \partial D\setminus \overline \Gamma$, 
and hence $\langle\oplus_{l=0}^{m-1}\langle B_l u , w^{(i)}_l \rangle_{\partial D}$ can be identified 
within $\tilde S_{A^*} ^{m-s,p'}(\hat X\setminus \overline D)$ as  
$$
\oplus_{l=0}^{m-1}\langle B_l u , v^{(i)}_l \rangle_{{\rm Gr}} \mbox{ with }
v^{(i)} = \tilde {\mathcal G} (\oplus_{\nu=0}^{m-1} \varphi w^{(i)}_\nu)^+.
$$
As $\varphi w^{(i)}$ vanishes on $U$, 
we see that $v^{(i)} \in S_{A^*} (\hat X\setminus (\partial D\setminus  U)) \cap \tilde S^{m-s,p'}_{A^*} 
(\hat X\setminus \overline D) \cap S^{m-s,p'}_{A^*} (D)$, $1\leq i \leq N$. Let us see that the system 
$\{v^{(i)}\}$ is linearly independent $\tilde S^{m-s,p'}_{A^*} 
(\hat X\setminus \overline D)$. Again, if we set   
$$
v (y) = 
\sum_{i=1}^{N} v^{(i)} \, \vec {C} _i
$$
with some vectors $\{\vec {C} _1 , \dots \vec {C} _N\} \subset {\mathbb R}^k$. As $X\setminus \overline D$ is 
connected then, by the Unique Continuation Property, if $v \equiv 0$ in $X\setminus \overline D$ then it 
vanishes on $X \setminus \overline\Gamma$. Now,  the jump theorems of 
Green type integrals $\tilde {\mathcal G} (\oplus_{\nu=0}^{m-1} \varphi w^{(i)}_\nu)$, see \eqref{eq.weak.jump}, 
we conclude that 
\begin{equation}\label{eq.jumps}
C^A_j (\tilde {\mathcal G} (\oplus_{\nu=0}^{m-1} \varphi w^{(i)}_\nu))^- - 
C^A_j (\tilde {\mathcal G} (\oplus_{\nu=0}^{m-1} \varphi w^{(i)}_\nu))^+  = w^{(i)}_j \mbox{ on }  \Gamma
\end{equation}
in the sense of weak boundary values. 
Therefore, 
$$
0 = \oplus_{j=0}^{m-1} C^A_j v^- - \oplus_{j=0}^{m-1} C^A_j v^+ = 
\sum_{i=1}^{N} w^{(i)} \, \vec {C} _i \mbox{ on }  \Gamma.
$$
As system \eqref{eq.vj.weak} is linearly independent in $\big(\oplus_{l=0}^{m-1}[\tilde B^{s-l-1/p,p} (\Gamma)]^k 
\big)^* $, we conclude that $\vec {C} _i = 0$ 
for all $1\leq i \leq N$, i.e. the system $\{v^{(i)} \}$ is linearly independent in 
$\tilde S^{m-s,p'}_{A^*} (\hat X\setminus \overline D)$. 
Thus, the statement follows from Theorem \ref{t.sparse.s.p.D.fund.vj}.
\end{proof}

Note that if $s>m-1+n/p$ then, by the Sobolev Embedding Theorem, $S_A^{s,p} (D) \subset [C^{m-1} (\overline D)]^k$ 
and  then $B_j u \in [C^{m-j-1} (\partial D)]^k$, i.e. we may calculate them point-wisely on $\Gamma$.
Then we define $ {\mathcal M}= {\mathcal M}_{kNm}  : S_A^{s,p} (D)\to {\mathbb R}^{kNm}$:
\begin{equation} \label{eq.M.points.Bj}
{\mathcal M} u= ( \oplus_{l=0}^{m-1} ( B_l u) (x^{(l,1)}) , \dots 
\oplus_{l=0}^{m-1}( B_l u) (x^{(l,N)}) \big) ,
\end{equation}
with a set $x^{(\cdot)}_N$ of some pair-wise distinct points $\{ x^{(l,j)} \}_{1\leq j \leq N \atop 
0 \leq l \leq m-1} \subset \overline \Gamma$.

\begin{cor} \label{c.sparse.s.p.D.fund.Bj} 
Let $s>m-1+n/p$ and ${\mathcal M}$ be given by \eqref{eq.M.points.Bj}.
Under the hypothesis of Proposition \ref{p.exist.s.p.D}, 
given pair-wise distinct points $\{ x^{(l,j)} \}_{1\leq j \leq N \atop 
0 \leq l \leq m-1}\subset \overline \Gamma$,
there is a sparse minimizer $u^{\sharp}_{kNm}$ to Problem \ref{pr.A.s.p.D} in the form  \eqref{eq.sparse.u}  
with  some points $\{ y^{(l,j)} \}_{1\leq j \leq N \atop 
0 \leq l \leq m-1} $ from $X \setminus D$ and $\{ \vec{C}^{(l,j)} \}_{1\leq j \leq N \atop 
0 \leq l \leq m-1} $  from ${\mathbb R}^k$. 
\end{cor}

\begin{proof} 
Follows from Corollary \ref{c.sparse.s.p.D.fund.vj.weak} 
in the same way, as 
Corollary \ref{c.sparse.s.p.D.fund.Bj.Dir}  
follows from Corollary 
\ref{c.sparse.s.p.D.fund.vj.weak.Dir}.
To see it we have to show that the composition of $\delta$-functionals, related to 
pair-wise distinct points $\{ x^{(l,j)} \}_{1\leq j \leq N \atop 
0 \leq l \leq m-1}\subset \overline \Gamma$, with action of differential operators
$B_j$, $0 \leq j \leq {m-1}$,  are linearly independent in 
$(S_A^{s,p} (D))^*$. 
Indeed, let 
$$
\sum_{j=1}^N  \sum_{l=0}^{m-1} c_{j,l} ( B_l u) (x^{(l,j)}) =0 \mbox{ for all } u \in S_A^{s,p} (D)
$$
with some real numbers $c_{j,l}$. Taking a test function $\varphi_{i_0,l_0} \in 
[C^\infty (\partial D)]^k$ supported in a neighbourhood $V(x^{(i_0, l_0)})$ of a fixed point 
$x^{(i_0,l_0)}$, satisfying $\{ x^{(i,l)}\}_{(i,l)\ne (i_0,l_0)} \not \in V(x^{(i_0,l_0)})$, and using 
the fact that Problem \ref{pr.Cau} is densely solvable we approximate
the jet $\oplus_{j=0}^{m-1} u_j$ with $u_j = 0$, $l\ne l_0$, $u_{l_0} = \varphi_{i_0,l_0}$ 
by a sequence $\{ \oplus_{l=0}^{m-1} B_l u^{(\nu )}\} $ with  $\{ u^{(\nu )} \} \subset S_A ^{s,p} (D)$. 
As we may alway choose 
$\varphi_{i_0,l_0} (x^{(i_0, l_0)}) \ne 0$ 
we see that 
$$
0=\lim_{\nu \to \infty}\sum_{j=1}^N  \sum_{l=0}^{m-1} c_{j,l} ( B_l u^{(\nu )} ) (x^{(j,l)}) = c_{j_0,l_0} \varphi_{i_0,l_0} 
(x^{(i_0, l_0)}) 
$$
and hence $c_{j_0,l_0} = 0$ with arbitrary $(j_0,l_0)$. Thus, the statement of this corollary follows
 from  Corollary \ref{c.sparse.s.p.D.fund.vj.weak}.
\end{proof}

In this case operator equation \eqref{eq.M.infty} is equivalent to the following: given 
the set of pairwise distinct points $\{ x^{(i)}\}_{i=1}^{\infty} \subset \overline D$, and a vector 
$\vec{h}^{(0)} \in {\mathfrak M}$, 
$$\vec{h}^{(0)}= ( \oplus_{l=0}^{m-1}  \vec{h}_{1,l}, \dots 
\oplus_{l=0}^{m-1}  \vec{h}_{N,l}, \dots), \, \vec{h}_{i,l} \in 
{\mathbb R}^k,
$$ 
 find   $u \in S_A ^{s,p} (D)$ such that 
\begin{equation}
\label{eq.M.infty.Cau}
\oplus_{l=0}^{m-1} B_l u  (x^{(i,l)}) = \oplus_{l=0}^{m/2-1} \vec{h}_{i,l} \mbox{ for all } i\in \mathbb N.
\end{equation}

\begin{cor} \label{c.sparse.Cau.infty} 
Let $s>m-1+n/p$, ${\mathcal M}$ be given by \eqref{eq.M.points.Bj}, the functional $F$ be given by 
\eqref{eq.F.norm} and $\lim\limits_{N\to +\infty} b_N =0$. 
Under the hypothesis of Proposition \ref{p.exist.s.p.D}, 
let the pair-wise distinct points $\{ x^{(i,l)}\}_{i\in \mathbb N \atop 0\leq l \leq m-1} \subset 
\overline \Gamma$ be chosen in such a way that for each $0\leq l\leq m-1$ the set $\{ x^{(i,l)}\}_{i\in \mathbb N}$  
be dense in a relatively open subset $U$ of $\Gamma$. Then 1) operator equation \eqref{eq.M.infty.Cau}
can not have more than one solution $u^{(0)}_\infty$ in $S_A ^{s,p} (D)$; 2) if equation \eqref{eq.M.infty.Cau}
has a solution $u^{(0)}_\infty$ in $S_A ^{s,p} (D)$, then the sequence of minimizers 
$\{u^{(0)}_{kNm}\}$ to Problem \ref{pr.A.s.p.D} converges weakly to it in $S_A ^{s,p} (D)$; 3) if 
$ \oplus_{l=0}^{m-1} \vec{h}_{i,l} = \oplus_{l=0}^{m-1}  u_l (x^{(i,l)})$, $i \in \mathbb N$, for some 
set $\oplus_{l=0}^{m-1} u_l  \in \oplus_{l=0}^{m-1} [\tilde B^{s-l-1/p,p} (\Gamma)]^k $, 
$i \in \mathbb N$, and 
Problem \ref{pr.Cau} admits a solution $u^{(0)}$ for the data $\oplus_{l=0}^{m-1} u_l $, then the  sequence of 
minimizers $\{u^{(0)}_{kNm}\}$ converges to it.  
\end{cor}

\begin{proof} According to Lemma \ref{l.M.infty} and Corollary \ref{c.infty}, 
in order to prove the first and the second statement  we have to show that 
 the  linear envelope of functionals $\{ f_{i,l}  \}_{i\in \mathbb N \atop 0\leq l \leq m-1}$ induced by 
the operators $\{ ( B_l u) (x^{(i,l)}) \}_{i\in \mathbb N \atop 0\leq l \leq m-1}$ on 
the space $S_A ^{s,p} (D)$
is everywhere dense in $(S^{s,p}_A (D))^*$.
But, as the space $S^{s,p}_A (D)$ is reflexive, this is equivalent to the fact that an 
element $u \in  S^{s,p}_A (D)$, vanishing on $\{ f_{i,l}  \}_{i\in \mathbb N \atop 0\leq l \leq m-1}$
is identically zero.
On the other hand, 
if points $\{ x^{(i,l)}\}_{i=1}^{\infty} \subset \overline \Gamma$ are dense 
in a relatively open subset $U$ of $\Gamma$ for each $l$, $0\leq l \leq m-1$,
 then for any solutions $u$, satisfying $\oplus_{l=0}^{m-1} B_j u (x^{(i,l)})=0$ 
for all $i \in \mathbb N$ we see that $\oplus_{l=0}^{m-1} B_j u =0$ on $U$. 
Hence, by the Uniqueness Theorem for the Cauchy Problem \ref{pr.Cau}
see \cite{ShTaLMS}, \cite[Theorem 10.3.5]{Tark36}, $u \equiv  0$  in $D$. 

Finally, if Problem \ref{pr.Cau} admits (a unique) solution $u^{(0)} \in S_A ^{s,p} (D)$ 
for the data $\oplus_{l=0}^{m-1} u_l \in \oplus_{l=0}^{m-1} [\tilde B^{s-l-1/p,p} (\Gamma)]^k$, then 
related operator equation \eqref{eq.M.infty.Dir} is solvable for the data 
$ \oplus_{l=0}^{m-1} \vec{h}_{i,l} = \oplus_{l=0}^{m-1}  u_l (x^{(i,l)})$, $i \in \mathbb N$, and $u^{(0)}$ 
is its unique solution.
\end{proof}

Of course, we may do the same in the weak sense. In this case operator equation \eqref{eq.M.infty} is equivalent to the following: given a set 
$$
w^{(\cdot)}_\infty  = \{ \oplus_{l=0}^{m-1} w^{(i)}_l \} _{i \in \mathbb N}  
 \subset (\oplus_{l=0}^{m-1}[\tilde B^{s-l-1/p,p} (\Gamma)]^k)^*
$$  and a vector 
$\vec{h}^{(0)} \in {\mathfrak M}$, 
$$\vec{h}^{(0)}= ( \oplus_{l=0}^{m-1}  \vec{h}_{1,l}, \dots 
\oplus_{l=0}^{m-1}  \vec{h}_{N,l}, \dots), \, \vec{h}_{i,l} \in 
{\mathbb R}^k,
$$ 
 find   $u \in S_A ^{s,p} (D)$ such that 
\begin{equation}
\label{eq.M.infty.Cau.weak}
\oplus_{l=0}^{m-1}\langle B_l u , w^{(i)}_l \rangle_{\partial D}= \oplus_{l=0}^{m-1} \vec{h}_{i,l} \mbox{ for all } i\in \mathbb N.
\end{equation}

\begin{cor} \label{c.sparse.Cau.infty.weak} 
Let ${\mathcal M}$ be given by \eqref{eq.M.vj.weak} and the functional $F$ be given by 
\eqref{eq.F.norm} and  $\lim\limits_{N\to +\infty} b_N =0$. 
Under the hypothesis of Proposition \ref{p.exist.s.p.D}, 
let $w^{(\cdot)}_\infty $ be a linearly independent system with its linear envelope 
being dense in the space 
$(\oplus_{l=0}^{m-1} 
[\tilde B^{s-l-1/p,p} (\Gamma)]^k)^* $. 
Then 1) the related equation \eqref{eq.M.infty.Cau.weak} 
can not have more than one solution $u^{(0)}_\infty$ in $S_A ^{s,p} (D)$; 2) if 
equation \eqref{eq.M.infty.Cau.weak} has a solution $u^{(0)}_\infty$ in $S_A ^{s,p} (D)$ then the sequence of 
minimizers $\{u^{(0)}_{kNm}\}$ to Problem \ref{pr.A.s.p.D} converges weakly in $S_A ^{s,p} (D)$ to it; 
3) if $ \oplus_{l=0}^{m-1} \vec{h}_{i,l} = \oplus_{l=0}^{m-1}\langle u_l  , w^{(1)}_l \rangle_{\partial D}$
 for some set 
$\oplus_{l=0}^{m-1} u_l  \in \oplus_{l=0}^{m-1} [\tilde B^{s-l-1/p,p} (\Gamma)]^k $, $i \in \mathbb N$, and Problem \ref{pr.Cau} admits a solution $u^{(0)}$ for the data $\oplus_{l=0}^{m-1} u_l $, then the  sequence of minimizers $\{u^{(0)}_{kNm}\}$ converges to it.  
\end{cor}

\subsection{Examples}
\label{s.ex} 

\begin{exmp} \label{ex.homogeneous}
Let $A$ be  an elliptic differential operator with constant coefficients on ${\mathbb R}^n$.  
As it is well known, both $A$ and its formal adjoint $A^*$ possess UCP because their solutions in a domain $D$ 
are actually real analytic there. In particular, the operator $A$ admits a fundamental solution 
of the convolution type, i.e. $\Phi (x,y) = \Phi (x-y)$, see, for instance, \cite[Theorem 8.5]{Rud_FA}, 
and moreover it can be chosen in a rational-exponential form.   

If, in addition, $A$ is  a homogeneous operator with constant coefficients, i.e.
$$
A=\sum_{|\alpha|= m}a_\alpha \partial^\alpha,
$$ 
then the standard fundamental solution for $A$ has the following form: 
$$
\Phi (x,y) = a\Big( \frac{x-y}{|x-y|}\Big) |x-y|^{m-n}+ b(x-y) \ln|x-y|, 
$$
where $a (\zeta)$ is a $(k\times k)$-matrix of real analytic function over the unit sphere in ${\mathbb R}^n$ and 
$b (\zeta)$ is a $(k\times k)$-matrix of homogeneous polynomials of degree $(m-n)$ that is identically zero for odd 
$n$, see \cite[\S 2.2.2]{Tark37}. Note that, in this case, for $m<n$ and a bounded domain $D$, a solution     
$v \in S_{A^*} ({\mathbb R}^n \setminus \overline D)$ is regular at the infinity with respect to $\Phi$ 
if it vanishes at the infinity: $\lim_{|x| \to + \infty} v(x) = 0$.
\end{exmp}

\begin{exmp} \label{ex.CR}
Let $A$ be the Cauchy-Riemann  operator $\overline \partial$ on 
on  ${\mathbb R}^2 \cong {\mathbb C}$. In the complex form we  have 
$$\overline \partial  = \frac{\partial}{\partial \overline z} = \frac{1}{2}
\Big(\frac{\partial}{\partial x} + \iota \frac{\partial}{\partial y} \Big),
$$ 
where $z=x_1+\iota x_2$, $\zeta=y_1+\iota y_2$ and $\iota$ is the imaginary unit. Of course, we  may easily 
decomplexify it to a $(2\times 2)$ system of PDE's. The space $S_{\overline \partial} (D)$ consists of holomorphic 
functions in domain $D$; in particular, the space contains polynomials of variable $z$ of any degree.  
Both $\overline \partial$ and its formal adjoint $\overline \partial^*$ possess UCP 
(even more strong Uniqueness Theorem holds for them). 
Then $\Phi (x,y)= \frac{1}{\pi (\zeta-z)}$ is the standard 
fundamental solution for $\overline \partial$ and the related second Green formula is just the standard Integral 
Cauchy Formula for holomorphic functions.  Again, a holomorphic function  $v$ on  ${\mathbb C} \setminus \overline D$ 
is regular at the infinity with respect to $\Phi$ if it vanishes at the infinity,  
see \cite{Grot2}, \cite{Koth2}, \cite{Silva}.

Since  a very strong uniqueness theorem holds true for holomorphic functions,  
then one may obtain a more strong statement on the determinant
$\det W_N (v^{(\cdot)}, y^{(\cdot)})$ than in Lemma \ref{l.non.degen.vj}. Namely, if
 $v^{(i)} (\zeta)= \frac{1}{2\pi \iota (\zeta-z^{(i)})} $, then, 
for fixed pair-wise distinct points $\{ z^{(1)},z^{(2)}, \dots z^{(N)}) \} \subset D$  the determinant 
$\det W_N (v^{(\cdot)},\zeta^{(\cdot)})$ does not vanish  
  if the points $\{ \zeta^{(1)},\zeta^{(2)}, \dots \zeta^{(N)} \} \subset {\mathbb C} \setminus \overline D $ are 
pair-wise distinct. 

The sparse minimizing with given points in a domain $D$ resembles the situation  with 
the approximation by 'simplest fractions', see \cite{DaDa}, closely related to the so-called 
Pad\'e approximation for holomorphic functions, see \cite{BG-M}.

The Cauchy problem for the Cauchy-Riemann system is known since T. Carleman \cite{Carl}, 
see also book \cite{Aiz} and its bibliography. It naturally appears in Aerodynamics and the Radio Physics.   
\end{exmp}

\begin{exmp} \label{ex.harm}
Let $A$ be the Laplace operator $\Delta = \sum_{j=1}^n \frac{\partial ^2}{\partial x_j^2}$  on ${\mathbb R}^n$; 
its solutions over a domain $D$ are harmonic functions there. In particular, the space of solutions 
contains homogeneous harmonic polynomials of any degree. The standard fundamental solution for $\Delta$ has 
the following form: 
$$
\Phi (x,y) = 
\left\{
\begin{array}{lll} \frac{|x-y|^{2-n}}{(2-n)\sigma_n} , & n> 2,\\
\frac{1}{2\pi} \ln{|x-y|}, & n=2, \\
\end{array}
\right.
$$
where $\sigma_n$ is the square of the unit sphere in ${\mathbb R}^n$. The related second Green formula is just the 
standard Green formula for harmonic functions with $B_0 =1 = C^\Delta_0$, $B_1 =
\frac{\partial }{\partial \nu} = C^\Delta_1$, where $\frac{\partial }{\partial \nu}$ is the normal 
derivative with respect to $\partial D$. 

 A harmonic function  $v$ on  ${\mathbb R}^n \setminus \overline D$, $n\geq 3$,
is regular at the infinity with respect to $\Phi$ if it vanishes at the infinity; for $n=2$ 
it is a function of type \eqref{eq.Groth.reg.inf}.
 
Note that the determinant of matrix 
\eqref{eq.W.points} may vanish on 'thin' sets in $X\setminus \overline D$. 
Indeed, if $n=3$ then $\Phi (x,y)= \frac{1}{4\pi}|x-y|^{-1}$. If $D$ is the unit ball then 
for $N=2$ and the points $x^{(1)} = (0,0,0)$, $x^{(2)} = (1/2,0,0)$, 
 $y^{(1)} = (2,2,0)$,  we have 
$$
v^{(1)} (y) = (4\pi |x^{(1)} - y|)^{-1} = (4\pi |y|)^{-1}   , \quad 
v^{(2)} (y) = (4\pi |x^{(2)} - y|)^{-1} .
$$
In particular,  
$$
|x^{(1)} - y^{(1)}| = 2\sqrt{2}, \, |x^{(2)} - y^{(1)}| = 5/2.
$$
Solving the equation 
$$
5/2 |x^{(1)} - y|   = 
2\sqrt{2}  |x^{(2)} - y| 
$$
we obtain for $y = y^{(2)}$:
\begin{equation} \label{eq.sphere}
 y_1^2 +  y_2^2 + y_3^2  +y_3-32/7 y_1 + 8 /7= 0 = (y_1-16/7)^2 + y_2^2 + y_3^2 +8 /7 - (16 /7)^2.
\end{equation}
Thus, the columns of the matrix $W_2 (v^{(\cdot)}, y^{(1)},y^{(2)})$ are proportional  on sphere 
\eqref{eq.sphere}  with respect to $y^{(2)}$ for fixed points $x^{(1)}, x^{(2)}, y^{(1)}$ and 
hence $\det W_N (y^{(1)}, y^{(2)})$ vanishes on it; in particular, 
$\det W_2 (y^{(1)}, y^{(2)})$ equals to zero at the point 
$y^{(2)}=(2,-2,0)$.

The Cauchy Problem for the Laplace equation is a classical example of an ill-posed problem 
of Mathematical Physics since J. Hadamard. 
In our context it reads as follows: 
given data $u_0 \in \tilde B^{s-1/p,p} (\Gamma)$, $u_1 \in \tilde B^{s-1- 1/p,p} (\Gamma)$
find $u \in W^{s,p} (D)$ such that 
\begin{equation*}
\left\{ 
\begin{array}{lccc}
\Delta u=0 & in & D, \\
 u  =  u_0  & on & \Gamma,  \\
\frac{\partial u}{\partial \nu}  =  u_1  & on & \Gamma,
\end{array}
\right.
\end{equation*}
see \cite{LvRShi}, \cite{Sh92A}, \cite{ShTaLMS}.

The Dirichlet Problem for the Laplace equation plays an essential role as a model example of 
a well-posed problem of Mathematical Physics. 
In our context it reads as follows: 
given datum $u_0 \in B^{s-1/p,p} (\partial D)$, 
find $u \in W^{s,p} (D)$ such that 
\begin{equation*}
\left\{ 
\begin{array}{lcc}
\Delta u=0 & in & D, \\
 u  =  u_0  & on & \partial D,  \\
\end{array}
\right.
\end{equation*}
see, for instance, \cite{EgShu}, \cite{LiMa72}, \cite{ShTaBDRK}.
\end{exmp}

%


\section{Conclusion}

Thus, we considered a minimization problem (with
finite-dimensional data) in Sobolev spaces of solutions to an  elliptic system 
in a bounded domain. Using the Grothendieck type dualities we prove theorems on the existence of sparse minimizer 
as a linear combination of a suitable fundamental solution to the elliptic system at different points 
in the complement of the domain with respect to the 'interior variable'. The case where 
the number of data passes to infinity were also discussed. Some applications 
to  problem related to neural networks, to the ill-posed Cauchy problem 
for elliptic operators  and to the well-posed Dirichlet problem for strongly elliptic operators were considered.

\bigskip

{\sc Acknowledgments.} 
This work was supported by the Krasnoyarsk Mathematical Center and financed by the Ministry of Science and Higher Education of the Russian Federation (Agreement No. 075-02-2025-1790).

%
%

%
%
%
%
%
%
%
%

\end{document}